\documentclass{amsart}
\usepackage{amssymb}
\usepackage{braket} 
\usepackage{mathrsfs}
\usepackage[all]{xy}
\usepackage{graphicx}

\newtheorem{thm}{Theorem}[section]
\newtheorem{cor}[thm]{Corollary}
\newtheorem{prop}[thm]{Proposition}
\theoremstyle{definition}
\newtheorem{dfn}[thm]{Definition}
\newtheorem{ex}[thm]{Example}

\newtheorem{lem}[thm]{Lemma}

\newtheorem{condition}{Condition}
\theoremstyle{remark}
\newtheorem{rem}[thm]{Remark}
\newcommand{\Ob}{\mathrm{Ob}}         
\newcommand{\id}{\mathrm{id}}         
\newcommand{\Id}{\mathrm{Id}}         
\newcommand{\ppr}{^{\prime}}          
\newcommand{\pprr}{^{\prime\prime}}   
\newcommand{\pro}{\mathrm{pr}}        
\newcommand{\op}{\mathrm{op}}         
\newcommand{\fa}{\forall}             
\newcommand{\am}{\amalg}              
\newcommand{\co}{\colon}              
\newcommand{\ci}{\circ}               
\newcommand{\iv}{^{-1}}               
\newcommand{\se}{\subseteq}           
\newcommand{\ti}{\times}              
\newcommand{\uas}{^{\ast}}            
\newcommand{\sas}{_{\ast}}            
\newcommand{\ems}{\emptyset}          
\newcommand{\lt}{\dashv}              

\newcommand{\Sett}{\mathit{Set}}    
\newcommand{\sett}{\mathit{set}}    

\newcommand{\lla}{\longleftarrow}     
\newcommand{\lra}{\longrightarrow}    
\newcommand{\tc}{\Rightarrow}         
\newcommand{\LR}{\Leftrightarrow}     

\newcommand{\al}{\alpha}         
\newcommand{\be}{\beta}          
\newcommand{\lam}{\lambda}       
\newcommand{\ups}{\upsilon}      
\newcommand{\sig}{\sigma}        
\newcommand{\ep}{\varepsilon}    
\newcommand{\thh}{\theta}        
\newcommand{\om}{\omega}         
\newcommand{\vp}{\varphi}        
\newcommand{\Om}{\Omega}         

\newcommand{\Bbbb}{\mathbb{B}}  
\newcommand{\Cbb}{\mathbb{C}}   
\newcommand{\Dbb}{\mathbb{D}}   
\newcommand{\Sbb}{\mathbb{S}}   
\newcommand{\Fcal}{\mathcal{F}} 
\newcommand{\afr}{\mathfrak{a}} 
\newcommand{\bfr}{\mathfrak{b}} 
\newcommand{\Afr}{\mathfrak{A}} 

\newcommand{\wt}{\widetilde}    
\newcommand{\ovl}{\overline}    
\newcommand{\ov}{\overset}      
\newcommand{\un}{\underset}     

\newcommand{\Ind}{\mathrm{Ind}}    
\newcommand{\Inf}{\mathrm{Inf}}    
\newcommand{\Res}{\mathrm{Res}}    
\newcommand{\Jnd}{\mathrm{Jnd}}    
\newcommand{\Inv}{\mathrm{Inv}}    
\newcommand{\Orb}{\mathrm{Orb}}    
\newcommand{\nm}{\vartriangleleft} 

\newcommand{\Gs}{{}_G\mathit{set}}           
\newcommand{\Hs}{{}_H\mathit{set}}           

\newcommand{\iog}{\iota^{(G)}}
\newcommand{\ioh}{\iota^{(H)}}
\newcommand{\prg}{\pro^{(G)}}
\newcommand{\prh}{\pro^{(H)}}
\newcommand{\ax}{\al(x)}        
\newcommand{\oc}{\althh\co\xg\to\yh}        
\newcommand{\ata}{(\althh)\uas}        
\newcommand{\atp}{(\althh)\pl}        
\newcommand{\atb}{(\althh)\bu}        
\newcommand{\Aa}{(A,\afr)}        
\newcommand{\Aap}{(A\ppr,\afr\ppr)}        
\newcommand{\Bb}{(B,\bfr)}        
\newcommand{\Bbp}{(B\ppr,\bfr\ppr)}        
\newcommand{\Aar}{(A\ov{\afr}{\to}X)}        
\newcommand{\Bbr}{(B\ov{\bfr}{\to}Y)}        
\newcommand{\XY}{X\un{Y}{\ti}}
\newcommand{\HG}{H\un{G}{\ti}}
\newcommand{\di}{^{\diamond}}
\newcommand{\eq}{\equiv}
\newcommand{\bu}{_{\bullet}}
\newcommand{\pl}{_+}
\newcommand{\althhd}{\frac{\al\di}{\thh\di}}
\newcommand{\ual}{^{\althh}}
\newcommand{\uth}{^{\thh}}
\newcommand{\uu}{(u_1/u_2)}

\newcommand{\pt}{\mathbf{1}}                     
\newcommand{\SIm}{\mathrm{SIm}}                  

\newcommand{\Mack}{\mathit{Mack}}             
\newcommand{\MackS}{\Mack(\Sbb)}              

\newcommand{\xg}{\frac{X}{G}}

\newcommand{\yh}{\frac{Y}{H}}
\newcommand{\yg}{\frac{Y}{G}}

\newcommand{\zk}{\frac{Z}{K}}

\newcommand{\wl}{\frac{W}{L}}

\newcommand{\pte}{\frac{\pt}{e}}
\newcommand{\ptg}{\frac{\pt}{G}}
\newcommand{\pth}{\frac{\pt}{H}}

\newcommand{\ptq}{\frac{\pt}{q}}

\newcommand{\althh}{\frac{\alpha}{\theta}}
\newcommand{\ulthh}{\frac{\ups_{\alpha}}{\theta}}
\newcommand{\althhp}{\frac{\alpha\ppr}{\theta\ppr}}
\newcommand{\bet}{\frac{\beta}{\tau}}
\newcommand{\betp}{\frac{\beta\ppr}{\tau\ppr}}

\numberwithin{equation}{section}

\begin{document}

\title[Partial Tambara structure on the Burnside biset functor]{Partial Tambara structure on the Burnside biset functor, induced from a derivator-like system of adjoint triplets}

\author{Hiroyuki NAKAOKA}
\address{Research and Education Assembly, Science and Engineering Area, Research Field in Science, Kagoshima University, 1-21-35 Korimoto, Kagoshima, 890-0065 Japan\ /\ LAMFA, Universit\'{e} de Picardie-Jules Verne, 33 rue St Leu, 80039 Amiens Cedex1, France}

\email{nakaoka@sci.kagoshima-u.ac.jp}
\urladdr{http://www.lamfa.u-picardie.fr/nakaoka/}


\thanks{This work is supported by JSPS KAKENHI Grant Numbers 25800022,\, 24540085.}
\thanks{The author wishes to thank the referee for his instructive advices.}

\begin{abstract}
In the previous article \lq\lq{\it A Mackey-functor theoretic interpretation of biset functors}", we have constructed the 2-category $\Sbb$ {\it of finite sets with variable finite group actions}, in which bicoproducts and bipullbacks exist. 
As shown in it, biset functors can be regarded as a special class of Mackey functors on $\Sbb$.

In this article, we equip $\Sbb$ with a system of adjoint triplets, which satisfies properties analogous to a derivator.
This system encodes the six operations for finite groups. As a corollary, we show that the associated Burnside rings satify analogous properties to a Tambara functor, in the context of biset functor theory.
\end{abstract}

\maketitle


\section{Introduction and Preliminaries}
In the previous article \cite{N_BisetMackey}, we have constructed the 2-category $\Sbb$ {\it of finite sets with variable finite group actions}. This 2-category admits finite bicoproducts and bipullbacks, which enable us to define the category $\MackS$ of Mackey functors on $\Sbb$. 
As shown in \cite{N_BisetMackey}, biset functors defined by Bouc (\cite{Bouc_Biset}) can be regarded as a special class of these Mackey functors.

In this article, we equip $\Sbb$ with a system of adjoint triplets, which satisfies properties analogous to the defining conditions for derivators.
This system encodes the six operations for finite sets with finite group actions, namely, induction, restriction, inflation, multiplicative induction, taking orbits and invariant parts (Example \ref{Ex5.1}).
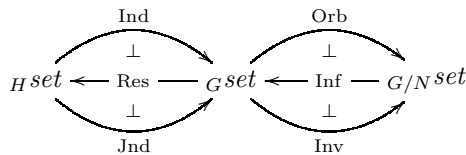
\[
\xy
(-26,0)*+{\Hs}="0";
(0,0)*+{\Gs}="2";
(26,0)*+{{}_{G/N}\sett}="4";
{\ar@{<-}@/_1.6pc/_{\Ind} "2";"0"};
{\ar@{<-}|*+{_{\Res}} "0";"2"};
{\ar@{<-}@/^1.6pc/^{\Jnd} "2";"0"};
{\ar@{<-}@/_1.6pc/_{\Orb} "4";"2"};
{\ar@{<-}|*+{_{\Inf}} "2";"4"};
{\ar@{<-}@/^1.6pc/^{\Inv} "4";"2"};
{\ar@{}|{\perp} (-13,5);(-13,3)};
{\ar@{}|{\perp} (-13,-3);(-13,-5)};
{\ar@{}|{\perp} (13,5);(13,3)};
{\ar@{}|{\perp} (13,-3);(13,-5)};
\endxy
\]
As a corollary, we show that the associated Burnside rings satisfy analogous properties to Tambara functors, in the context of biset functor theory.

\smallskip

Throughout this article, for a finite group $G$, let $e\in G$ denote its unit element. Abbreviately, we denote the trivial group by the same symbol $e$.

\medskip

The strict 2-category $\Sbb$ {\it of finite sets with variable finite group actions}, is defined as follows. (For generalities of 2-categories, see \cite{Borceux},\cite{MacLane}.)
\begin{dfn}\label{DefS}$($\cite[Definition 2.2.12]{N_BisetMackey}$)$
2-category $\Sbb$ is defined by the following.
\begin{enumerate}
\item[{\rm (0)}] A 0-cell is a pair of a finite group $G$ and a finite $G$-set $X$. We denote this pair by $\xg$.
\item[{\rm (1)}] For any pair of 0-cells $\xg$ and $\yh$, a morphism $\althh\co \xg\to\yh$ is a pair of a map $\al\co X\to Y$ and a family of maps $\{\thh_x\co G\to H \}_{x\in X}$ satisfying
\begin{itemize}
\item[{\rm (i)}] $\al(gx)=\thh_x(g)\al(x)$ 
\item[{\rm (ii)}] $\thh_x(gg\ppr)=\thh_{g\ppr x}(g)\thh_x(g\ppr)$
\end{itemize}
for any $x\in X$ and any $g,g\ppr\in G$. $\al$ is called the {\it base map} of 1-cell $\althh$. $\thh$ is called the {\it acting part}.

For any sequence of 1-cells
\[ \xg\ov{\althh}{\lra}\yh\ov{\bet}{\lra}\zk, \]
their composition $(\bet)\ci(\althh)=\frac{\be\ci\al}{\tau\ci\thh}$ is defined by
\begin{itemize}
\item[-] $\be\ci\al\co X\to Z$ is the usual composition of maps of sets,
\item[-] $\tau\ci\thh$ is defined by
\[ (\tau\ci\thh)_x=\tau_{\ax}\ci\thh_x \]
for any $x\in X$.
\end{itemize}
If there is a group homomorphism $f$ such that $\thh_x=f\ (\fa x\in X)$, then we denote $\althh$ by $\frac{\al}{f}$. In particular if $f=\id_G$, then we denote $\frac{\al}{\id_G}$ simply by $\frac{\al}{G}$. For any 0-cell $\xg$, the identity 1-cell is given by $\id_{\xg}=\frac{\id_X}{G}\co\xg\to\xg$.
\item[{\rm (2)}] For any pair of 1-cells $\althh,\althhp\co\xg\to\yh$, a 2-cell $\ep\co\althh\tc\althhp$ is a family of elements $\{ \ep_x\in H\}_{x\in X}$ satisfying
\begin{itemize}
\item[{\rm (i)}] $\al\ppr(x)=\ep_x\al(x) $,
\item[{\rm (ii)}] $\ep_{gx}\thh_x(g)\ep_x\iv=\thh\ppr_x(g)$
\end{itemize}
for any $x\in X$ and $g\in G$.
\end{enumerate}
If we are given a sequence of 2-cells
\[
\xy
(-14,0)*+{\xg}="0";
(14,0)*+{\yh}="2";
{\ar@/^2.0pc/^{\althh} "0";"2"};
{\ar|*+{_{\althhp}} "0";"2"};
{\ar@/_2.0pc/_{\frac{\al\pprr}{\thh\pprr}} "0";"2"};
{\ar@{=>}^{\ep} (0,6);(0,3)};
{\ar@{=>}^{\ep\ppr} (0,-3);(0,-6)};
\endxy
\]
then their vertical composition $\ep\ppr\cdot\ep\co \althh\tc \frac{\al\pprr}{\thh\pprr}$ is defined by
\[ (\ep\ppr\cdot\ep)_x=\ep\ppr_x\ep_x\quad(\fa x\in X). \]
Any 2-cell is invertible with respect to this vertical composition.
\end{dfn}

\begin{dfn}\label{DefHoriz}
Let $\xg\ov{\althh}{\lra}\yh\ov{\bet}{\lra}\zk$ be a sequence of 1-cells.
\begin{enumerate}
\item For a 2-cell
\[
\xy
(-14,0)*+{\xg}="0";
(14,0)*+{\yh}="2";
{\ar@/^1.2pc/^{\althh} "0";"2"};
{\ar@/_1.2pc/_{\althhp} "0";"2"};
{\ar@{=>}^{\ep} (0,2);(0,-2)};
\endxy,
\]
define $(\bet)\ci\ep\co(\bet)\ci(\althh)\tc(\bet)\ci(\althhp)$ by
\begin{equation}\label{EqHor1}
((\bet)\ci\ep)_x=\tau_{\ax}(\ep_x)\quad(\fa x\in X).
\end{equation}
\item For a 2-cell
\[
\xy
(-14,0)*+{\yh}="0";
(14,0)*+{\zk}="2";
{\ar@/^1.2pc/^{\bet} "0";"2"};
{\ar@/_1.2pc/_{\betp} "0";"2"};
{\ar@{=>}^{\delta} (0,2);(0,-2)};
\endxy,
\]
define $\delta\ci(\althh)\co(\bet)\ci(\althh)\tc(\betp)\ci(\althh)$ by
\begin{equation}\label{EqHor2}
(\delta\ci(\althh))_x=\delta_{\ax}\quad(\fa x\in X).
\end{equation}
\end{enumerate}
\end{dfn}

Horizontal composition is denoted by \lq\lq$\ci$", while \lq\lq$\cdot$" denotes vertical composition. For example, for any diagram
\[
\xy
(-28,0)*+{\xg}="0";
(0,0)*+{\yh}="2";
(28,0)*+{\zk}="4";
{\ar@/^1.2pc/^{\althh} "0";"2"};
{\ar@/_1.2pc/_{\althhp} "0";"2"};
{\ar@/^1.2pc/^{\bet} "2";"4"};
{\ar@/_1.2pc/_{\betp} "2";"4"};
{\ar@{=>}^{\ep} (-14,2);(-14,-2)};
{\ar@{=>}^{\delta} (14,2);(14,-2)};
\endxy
\]
in $\Sbb$, we have an equality
\[ (\delta\ci\althhp)\cdot(\bet\ci\ep)=(\betp\ci\ep)\cdot(\delta\ci\althh). \]

In $\Sbb$, we can deal with both group homomorphisms and equivariant maps, in the following way.
\begin{ex}\label{ExOf1cell}
$\ $
\begin{enumerate}
\item Any homomorphism of finite groups $f\co G\to H$ induces a 1-cell $\frac{\pt}{f}\co\ptg\to\pth$. Here, $\pt\co\pt\to\pt$ denotes the unique map between one-point set $\pt$ with trivial group action.
\item For a fixed finite group $G$, a $G$-map $\al\co X\to Y$ induces a 1-cell $\frac{\al}{G}\co\xg\to\yg$. We call this type of 1-cell $G$-{\it equivariant}, or simply {\it equivariant}.
\end{enumerate}
\end{ex}

\begin{dfn}\label{DefAdjEq}
Let $\althh\co\xg\to\yh$ be a 1-cell.
\begin{enumerate}
\item $\althh$ is an {\it equivalence} if there is a 1-cell $\bet\co\yh\to\xg$ and 2-cells
\[ \rho\co\bet\ci\althh\tc\id_{\xg},\ \ \lam\co\althh\ci\bet\tc\id_{\yh}. \]
$\bet$ is called a {\it quasi-inverse} of $\althh$.
\item $\althh$ is an {\it adjoint equivalence} if there is a 1-cell $\bet$ and 2-cells $\rho,\lam$ as above, which moreover satisfy
\[ \althh\ci\rho=\lam\ci\althh,\ \ \rho\ci\bet=\bet\ci\lam \]
in the diagram
\[
\xy
(-40,0)*+{\xg}="0";
(-20,0)*+{\yh}="2";
(0,0)*+{\xg}="4";
(20,0)*+{\yh}="6";
(40,0)*+{\xg}="8";
{\ar^{\althh} "0";"2"};
{\ar^{\bet} "2";"4"};
{\ar_{\althh} "4";"6"};
{\ar^{\bet} "6";"8"};
{\ar@/_1.8pc/_{\id} "0";"4"};
{\ar@/^1.8pc/^{\id} "2";"6"};
{\ar@/_1.8pc/_{\id} "4";"8"};
{\ar@{=>}^{\rho} (-20,-3);(-20,-6)};
{\ar@{=>}^{\lam} (0,3);(0,6)};
{\ar@{=>}^{\rho} (20,-3);(20,-6)};
\endxy.
\]
\end{enumerate}
\end{dfn}

\medskip
\begin{rem}$($\cite[Remark 2.2.14]{N_BisetMackey}$)$
For any 1-cell $\althh\co\xg\to\yh$ in $\Sbb$, the following are equivalent.
\begin{enumerate}
\item $\althh$ is an equivalence.
\item $\althh$ is an adjoint equivalence.
\end{enumerate}
For this reason, in the later argument, we simply call it an {\it equivalence}.
\end{rem}

\begin{dfn}\label{DefInd}
Let $\iota\co H\hookrightarrow G$ be a monomorphism of groups.
For any $X\in\Ob(\Hs)$, we define $\Ind_{\iota}X\in\Ob(\Gs)$ by
\[ \Ind_{\iota}X=G\un{H}{\ti}X=(G\times X)/\sim, \]
where the equivalence relation $\sim$ is defined by
\begin{itemize}
\item[-] $(\xi,x)$ and $(\xi\ppr,x\ppr)$ in $G\times X$ are equivalent if there exists $h\in H$ satisfying
\[ \xi=\xi\ppr\iota(h),\ \ x\ppr=hx. \]
\end{itemize}
We denote the equivalence class of $(\xi,x)$ by $[\xi,x]\in\Ind_{\iota}X$. The $G$-action on $\Ind_{\iota}X$ is defined by
\[ g[\xi,x]=[g\xi,x] \]
for any $g\in G$ and $[\xi,x]\in\Ind_{\iota}X$.
\end{dfn}

The following has been shown in \cite{N_BisetMackey}.
\begin{prop}\label{PropIndEquiv}$($\cite[Proposition 3.1.2]{N_BisetMackey}$)$
Let $\iota\co H\hookrightarrow G$ be a monomorphism of groups.
For any $X\in\Ob(\Hs)$, if we define a map $\ups\co X\to\Ind_{\iota}X$ by
\[ \ups(x)=[e,x]\quad(\fa x\in X), \]
then the 1-cell
\[ \frac{\ups}{\iota}\co\frac{X}{H}\to \frac{\Ind_{\iota}X}{G} \]
becomes an equivalence. We call this an {\it $\Ind$-equivalence}.
\end{prop}

\begin{dfn}\label{Def2Coproduct}
Let $\Cbb$ be a strict 2-category with invertible 2-cells.
For any pair of 0-cells $A_1$ and $A_2$ in $\Cbb$, their {\it bicoproduct} $(A_1\am A_2,\iota_1,\iota_2)$ is defined to be a triplet of 0-cell $A_1\am A_2$ and 1-cells
\[ A_1\ov{\iota_1}{\lra}A_1\am A_2\ov{\iota_2}{\lla}A_2, \]
satisfying the following conditions.
\begin{itemize}
\item[{\rm (i)}]
For any 0-cell $X$ and 1-cells $f_i\co A_i\to X\ (i=1,2)$, there exists a 1-cell $f\co A_1\am A_2\to X$ and 2-cells $\xi_i\co f\circ \iota_i\tc f_i\ (i=1,2)$ as in the following diagram.
\[
\xy
(0,-6)*+{X}="0";
(-18,10)*+{A_1}="2";
(0,10)*+{A_1\am A_2}="4";
(18,10)*+{A_2}="6";
{\ar_{f_1} "2";"0"};
{\ar_{f} "4";"0"};
{\ar^{f_2} "6";"0"};
{\ar^(0.4){\iota_1} "2";"4"};
{\ar_(0.4){\iota_2} "6";"4"};
{\ar@{=>}_{\xi_1} (-4,6);(-7.5,3)};
{\ar@{=>}^{\xi_2} (4,6);(7.5,3)};
\endxy
\]

\item[{\rm (ii)}]
For any triplets $(f,\xi_1,\xi_2)$ and $(f\ppr,\xi_1\ppr,\xi_2\ppr)$ as in {\rm (i)}, there exists a unique 2-cell $\eta\co f\tc f\ppr$ such that $\xi_i\ppr\cdot(\eta\circ \iota_i)=\xi_i\ (i=1,2)$, namely, the following diagram of 2-cells is commutative for $i=1,2$.
\[
\xy
(-10,6)*+{f\circ \iota_i}="0";
(10,6)*+{f\ppr\circ \iota_i}="2";
(0,-7)*+{f_i}="4";
(0,9)*+{}="6";
{\ar@{=>}^{\eta\circ \iota_i} "0";"2"};
{\ar@{=>}_{\xi_i} "0";"4"};
{\ar@{=>}^{\xi_i\ppr} "2";"4"};
{\ar@{}|\circlearrowright"6";"4"};
\endxy
\]
\end{itemize}
\end{dfn}

\begin{dfn}\label{Def2Pullback}
Let $\Cbb$ be a strict 2-category with invertible 2-cells.
For any 0-cells $A_1,A_2,B$ and 1-cells $f_i\co A_i\to B\ (i=1,2)$, {\it bipullback} of $f_1$ and $f_2$ is defined to be a quartet $(A_1\times_BA_2,\pi_1,\pi_2,\kappa)$ 
as in the diagram
\[
\xy
(-9,6)*+{A_1\times_BA_2}="0";
(9,6)*+{A_1}="2";
(-9,-6)*+{A_2}="4";
(9,-6)*+{B}="6";
{\ar^(0.6){\pi_1} "0";"2"};
{\ar_{\pi_2} "0";"4"};
{\ar^{f_1} "2";"6"};
{\ar_{f_2} "4";"6"};
{\ar@{<=}_{\kappa} (-2,-2);(2,2)};
\endxy
,
\]
which satisfies the following conditions.
\begin{itemize}
\item[{\rm (i)}]
For any diagram in $\Cbb$
\[
\xy
(-8,6)*+{X}="0";
(8,6)*+{A_1}="2";
(-8,-6)*+{A_2}="4";
(8,-6)*+{B}="6";
{\ar^{g_1} "0";"2"};
{\ar_{g_2} "0";"4"};
{\ar^{f_1} "2";"6"};
{\ar_{f_2} "4";"6"};
{\ar@{<=}^{\ep} (-2,-2);(2,2)};
\endxy
,
\]
there exist $g,\xi_1,\xi_2$ as in the diagram
\[
\xy
(-22,16)*+{X}="-2";
(-8,6)*+{A_1\times_BA_2}="0";
(8,6)*+{A_1}="2";
(-8,-7)*+{A_2}="4";
(8,-7)*+{B}="6";
{\ar^{g} "-2";"0"};
{\ar@/^1.24pc/^{g_1} "-2";"2"};
{\ar@/_1.24pc/_(0.68){g_2} "-2";"4"};
{\ar_(0.66){\pi_1} "0";"2"};
{\ar^{\pi_2} "0";"4"};
{\ar^{f_1} "2";"6"};
{\ar_{f_2} "4";"6"};
{\ar@{<=}_{\kappa} (-1.5,-3);(2.5,1)};
{\ar@{=>}^{\xi_2} (-12,3);(-16,-1)};
{\ar@{=>}_{\xi_1} (-8,9);(-6,14)};
\endxy
,
\]
satisfying $\ep\cdot(f_1\ci\xi_1)=(f_2\ci\xi_2)\cdot(\kappa\ci g)$, namely making the following diagram of 2-cells commutative.
\[
\xy
(-12,6)*+{f_1\ci\pi_1\ci g}="0";
(12,6)*+{f_2\ci\pi_2\ci g}="2";
(-12,-6)*+{f_1\ci g_1}="4";
(12,-6)*+{f_2\ci g_2}="6";
{\ar@{=>}^{\kappa\ci g} "0";"2"};
{\ar@{=>}_{f_1\ci\xi_1} "0";"4"};
{\ar@{=>}^{f_2\ci\xi_2} "2";"6"};
{\ar@{=>}_{\ep} "4";"6"};
{\ar@{}|\circlearrowright"0";"6"};
\endxy
\]
\item[{\rm (ii)}]
For any triplets $(g,\xi_1,\xi_2)$ and $(g\ppr,\xi_1\ppr,\xi_2\ppr)$ as in {\rm (i)}, there exists a unique 2-cell $\zeta\co g\tc g^{\prime}$ which satisfies $\xi_i\ppr\cdot(\pi_i\ci\zeta)=\xi_i\ (i=1,2)$.
\end{itemize}
\end{dfn}

Bicoproducts and bipullbacks are uniquely determined up to an equivalence.

Existence of bicoproducts and bipullbacks in $\Sbb$ are shown in \cite{N_BisetMackey}, as in the following propositions.
\begin{prop}\label{Prop2CoprodVari}$($\cite[Proposition 3.2.15]{N_BisetMackey}$)$
Let $\xg$ and $\yh$ be any pair of 0-cells in $\Sbb$. Denote the monomorphisms
\begin{eqnarray*}
&G\to G\times H\ ; \ g\mapsto (g,e)&\\
&H\to G\times H\ ; \ h\mapsto (e,h)&
\end{eqnarray*}
by $\iog$ and $\ioh$ respectively, and denote the natural maps
\begin{eqnarray*}
&X\to\Ind_{\iog}X\am\Ind_{\ioh}Y\ ;\ x\mapsto [e,x]\in\Ind_{\iog}X&\\
&Y\to\Ind_{\iog}X\am\Ind_{\ioh}Y\ ;\ y\mapsto [e,y]\in\Ind_{\ioh}Y&
\end{eqnarray*}
by $\ups_X$ and $\ups_Y$.
Then
\[ \xg\ov{\frac{\ups_X}{\iog}}{\lra}\frac{\Ind_{\iog}X\am\Ind_{\ioh}Y}{G\times H}\ov{\frac{\ups_Y}{\ioh}}{\lla}\yh \]
gives a bicoproduct of $\xg$ and $\yh$ in $\Sbb$.
\end{prop}

\begin{prop}\label{Prop2Pullback}$($\cite[Proposition 3.2.17]{N_BisetMackey}$)$
Let $\althh\co\xg\to\zk$ and $\bet\co\yh\to\zk$ be any pair of 1-cells in $\Sbb$.
Denote the natural projection homomorphisms by
\[ \prg\co G\times H\to G,\ \ \prh\co G\times H\to H. \]
If we
\begin{itemize}
\item[-] put $F=\{(x,y,k)\in X\times Y\times K\mid \be(y)=k\ax \}$, and put
\begin{eqnarray*}
&\wp_X\co F\to X\ ;\ (x,y,k)\mapsto x,&\\
&\wp_Y\co F\to Y\ ;\ (x,y,k)\mapsto y,&
\end{eqnarray*}
\item[-] equip $F$ with a $G\times H$-action
\begin{eqnarray*}
&(g,h)(x,y,k)=(gx,hy,\tau_{\be(b)}(h) k\thh_{\ax}(g)\iv)&\\
&(\fa (g,h)\in G\times H,\ \ \fa (x,y,k)\in F),&
\end{eqnarray*}
\item[-] define a 2-cell $\kappa\co\althh\ci\frac{\wp_X}{\prg}\tc\bet\ci\frac{\wp_Y}{\prh}$ by
\[ \kappa_{(x,y,k)}=k, \]
\end{itemize}
then the diagram
\[
\xy
(-12,7)*+{\frac{F}{G\times H}}="0";
(12,7)*+{\xg}="2";
(-12,-6)*+{\yh}="4";
(12,-6)*+{\zk}="6";
{\ar^(0.52){\frac{\wp_X}{\prg}} "0";"2"};
{\ar_{\frac{\wp_Y}{\prh}} "0";"4"};
{\ar^{\althh} "2";"6"};
{\ar_{\bet} "4";"6"};
{\ar@{<=}^{\kappa} (-2.5,-1.5);(2.5,2.5)};
\endxy
\]
gives a bipullback in $\Sbb$.
\end{prop}

Thus the 2-category $\Sbb$ has the following.
\begin{itemize}
\item[-] Initial 0-cell $\ems=\frac{\ems}{e}$, and terminal 0-cell $\pte$. Here $e$ denotes the trivial group.
\item[-] Any finite bicoproduct.
\item[-] Any bipullback.
\end{itemize}

\begin{dfn}\label{DefStabsurj}
A 1-cell $\oc$ is called {\it stab-surjective}, if the following conditions are satisfied.
\begin{itemize}
\item[{\rm (i)}] $Y=H\al(X)$ holds.
\item[{\rm (ii)}] If $x,x\ppr\in X$ and $h,h\ppr\in H$ satisfy $h\al(x)=h\ppr\al(x\ppr)$, then there exists $g\in G$ which satisfies $x\ppr=gx$ and $h=h\ppr\thh_x(g)$.
\end{itemize}
\end{dfn}

Properties of stab-surjective 1-cells are as follows (\cite{N_BisetMackey}).
\begin{prop}\label{PropStabsurj}$($\cite[section 4.1]{N_BisetMackey}$)$
The following holds for the stab-surjectivity.
\begin{enumerate}
\item If $\althh$ is an equivalence, then $\althh$ is stab-surjective.
\item Stab-surjectivity is closed under equivalences by 2-cells. Namely, if there exists a 2-cell $\ep\co\althh\tc\althhp$ and if $\althh$ is stab-surjective, then so is $\althhp$.
\item Stab-surjectivity is closed under compositions of 1-cells. Namely, if $\xg\ov{\althh}{\lra}\yh\ov{\bet}{\lra}\zk$ is a sequence of stab-surjective 1-cells, then so is $\bet\ci\althh$.
\item Stab-surjectivity is stable under bipullbacks. Namely, if
\[
\xy
(-8,6)*+{\wl}="0";
(8,6)*+{\xg}="2";
(-8,-6)*+{\yh}="4";
(8,-6)*+{\zk}="6";
{\ar^{\frac{\gamma}{\mu}} "0";"2"};
{\ar_{\frac{\delta}{\nu}} "0";"4"};
{\ar^{\althh} "2";"6"};
{\ar_{\bet} "4";"6"};
{\ar@{<=}^{\ep} (-2,-2);(2,2)};
\endxy
\]
is a bipullback in $\Sbb$ and if $\bet$ is stab-surjective, then so is $\frac{\gamma}{\mu}$.
\end{enumerate}
\end{prop}

\begin{dfn}\label{DefSIm}
Let $\oc$ be any 1-cell in $\Sbb$.
\begin{enumerate}
\item Define $\SIm(\althh)\in\Ob(\Hs)$ by
\[ \SIm(\althh)=\HG X=(H\times X)/\sim, \]
where the relation $\sim$ is defined as follows.
\begin{itemize}
\item[-] $(\eta,x),(\eta\ppr,x\ppr)\in H\times X$ are equivalent if there exists $g\in G$ satisfying
\[ x\ppr=gx\ \ \ \text{and}\ \ \ \eta=\eta\ppr\thh_x(g). \]
\end{itemize}
We denote the equivalence class of $(\eta,x)$ by $[\eta,x]$. The $H$-action on $\SIm(\althh)$ is given by
\[ h[\eta, x]=[h\eta,x]. \]
We call $\SIm(\althh)$ the {\it stabilizerwise image} of $\althh$.
\item Define a map $\ups_{\al}\co X\to\SIm(\althh)$ by
\[ \ups_{\al}(x)=[e,x]\quad(\fa x\in X), \]
then $\ulthh\co\xg\to \frac{\SIm(\althh)}{H}$ becomes a stab-surjective 1-cell.
\end{enumerate}
\end{dfn}

\begin{rem}
$\SIm(\althh)$ essentially depends only on the acting part $\thh$.
\end{rem}

The following has been shown in \cite{N_BisetMackey}.
\begin{prop}\label{PropSIm}$($\cite[Proposition 4.2.6]{N_BisetMackey}$)$
For any 1-cell $\oc$, if we define a map $\wt{\al}$ by
\[ \wt{\al}\co\SIm(\althh)\to Y\ ;\ [\eta,x]\mapsto\eta\al(x), \]
then the following diagram in $\Sbb$ becomes commutative.
\[
\xy
(-20,0)*+{\xg}="0";
(0,8)*+{\frac{\SIm(\althh)}{H}}="2";
(0,-6)*+{}="3";
(20,0)*+{\frac{Y}{H}}="4";
{\ar^(0.46){\ulthh} "0";"2"};
{\ar^(0.54){\frac{\wt{\al}}{H}} "2";"4"};
{\ar@/_0.8pc/_{\althh} "0";"4"};
{\ar@{}|\circlearrowright "2";"3"};
\endxy
\]
We call this the {\it $\mathit{SIm}$-factorization} of $\althh$.
\end{prop}

The following ensures the uniqueness of the $\SIm$-factorization, up to equivariant isomorphisms (\cite{N_BisetMackey}).
\begin{rem}\label{PropSImFactorSystem}$($\cite[Proposition 4.2.7]{N_BisetMackey}$)$
Let $\oc$ be any 1-cell, and let $\althh=\frac{\wt{\al}}{H}\ci\ulthh$ be its $\SIm$-factorization. If there is another factorization of $\althh$ as
\[
\xy
(-20,0)*+{\xg}="0";
(0,8)*+{\frac{W}{H}}="2";
(0,-6)*+{}="3";
(20,0)*+{\yh}="4";
{\ar^(0.46){\bet} "0";"2"};
{\ar^(0.54){\frac{\gamma}{H}} "2";"4"};
{\ar@/_0.8pc/_{\althh} "0";"4"};
{\ar@{=>}^{\ep} (0,3.4);(0,-1)};
\endxy
\]
in which $\bet$ is stab-surjective, then the $H$-map
\[ \om\co\SIm(\althh)\to W\ ;\ [\eta,x]\mapsto\eta\ep_x\be(x) \]
gives an isomorphism satisfying $\gamma\ci\om=\wt{\al}$, and $\ep$ gives a 2-cell $\ep\co\bet\tc\frac{\om}{H}\ci\ulthh$. As a diagram, this is depicted as follows.
\[
\xy
(-11,8)*+{\xg}="0";
(11,8)*+{\frac{W}{H}}="2";
(-11,-7)*+{\frac{\SIm(\althh)}{H}}="4";
(11,-7)*+{\yh}="6";
{\ar^{\bet} "0";"2"};
{\ar_{\ulthh} "0";"4"};
{\ar^{\frac{\gamma}{H}} "2";"6"};
{\ar_{\frac{\wt{\al}}{H}} "4";"6"};
{\ar_{\frac{\om}{H}} "4";"2"};
{\ar@{=>}_(0.3){\ep} (-4,4);(-6.5,-0.5)};
{\ar@{}|\circlearrowright (8,0);(5,-5)};
\endxy
\]
\end{rem}

\section{Bifunctor $\Bbbb$}
Let $\mathrm{CAT}$ denote the 2-category of categories, whose 1-cells are functors, and 2-cells are natural transformations. In this section, we give a bifunctor $\Bbbb\co\Sbb^{\op}\to\mathrm{CAT}$, which associates $\Bbbb(\xg)=\Gs/X$ to any 0-cell $\xg$ in $\Sbb$.
\begin{rem}
As in \cite{Groth}, to avoid a set-theoretical difficulty to deal with \lq\lq 2-category of categories", we use the terminology {\it bifunctor} to mean simply a correspondence
\begin{eqnarray*}
\text{0-cell}\ \xg&\mapsto& \text{category}\ \Bbbb(\xg)=\Gs/X,\\
\text{1-cell}\ \oc&\mapsto& \text{functor}\ \Bbbb(\althh)=\ata\co\Hs/Y\to\Gs/X,\\
\text{2-cell}\ \ep\co\althh\tc\althhp&\mapsto& \text{natural transformation}\ \Bbbb(\ep)\co\Bbbb(\althh)\tc\Bbbb(\althhp)
\end{eqnarray*}
with natural isomorphisms
\[ \Bbbb(\bet\ci\althh)\cong\Bbbb(\althh)\ci\Bbbb(\bet),\quad \Bbbb(\id_{\xg})\cong\Id_{\Bbbb(\xg)}, \]
which, when we take isomorphism classes $\Afr(\xg)=\Ob(\Gs/X)/\cong$ at each $\xg$, becomes functorial (Definition \ref{DefOmega}, Corollary \ref{CorOmega}).
\end{rem}

\begin{dfn}\label{DefAst}
Let $\oc$ be a 1-cell in $\Sbb$. Define a functor $\ata\co\Hs/Y\to\Gs/X$ as follows.
\begin{enumerate}
\item For $\Bb=\Bbr\in\Ob(\Hs/Y)$, define $\ata\Bb=(\XY B\ov{p_X}{\lra}X)$ by the following.
\begin{itemize}
\item[{\rm (i)}] $\XY B=\{(x,b)\mid \al(x)=\bfr(b)\}$ is the usual fibered product of sets.
\item[{\rm (ii)}] $G$-action on $\XY B$ is given by
\[ g(x,b)=(gx,\thh_x(g)b)\quad (\fa g\in G, (x,b)\in\XY B). \]
\item[{\rm (iii)}] $p_X\co \XY B\to X$ is the projection onto $X$, which is obviously $G$-equivariant.
\end{itemize}
\item For a morphism $f\co\Bb\to\Bbp$ in $\Hs/Y$, define $\ata(f)$ by
\[ \ata (f)\co\XY B\to \XY B\ppr\ ;\ (x,b)\mapsto (x,f(b)). \]
\end{enumerate}
\end{dfn}

\begin{rem}\label{RemAst}
For a $G$-equivariant 1-cell $\frac{\al}{G}\co\xg\to\yg$, the functor $(\frac{\al}{G})\uas\co\Gs/Y\to\Gs/X$ agrees with the usual pullback functor
\[ \al\uas=\XY -\co\Gs/Y\to\Gs/X. \]
\end{rem}

\begin{prop}\label{PropAst}
Let $\oc$ be any 1-cell, and let $\Bb$ be any object in $\Hs/Y$. Then the commutative diagram
\begin{equation}\label{PBEquiv}
\xy
(-9,8)*+{\xg}="0";
(9,8)*+{\frac{\XY B}{G}}="2";
(-9,-8)*+{\yh}="4";
(9,-8)*+{\frac{B}{H}}="6";
{\ar_{\frac{p_X}{G}} "2";"0"};
{\ar_{\althh} "0";"4"};
{\ar^{\althhd} "2";"6"};
{\ar^{\frac{\bfr}{H}} "6";"4"};
{\ar@{}|\circlearrowright "0";"6"};
\endxy
\end{equation}
where $\althhd$ is defined by
\begin{eqnarray*}
\al\di(x,b)=b,\quad \thh\di_{(x,b)}=\thh_x
\end{eqnarray*}
for any $(x,b)\in \XY B$, gives a bipullback in $\Sbb$.
\end{prop}
\begin{proof}
Remark that we have a bipullback
\[
\xy
(-10,7)*+{\xg}="0";
(10,7)*+{\frac{F}{G\ti H}}="2";
(-10,-7)*+{\yh}="4";
(10,-7)*+{\frac{B}{H}}="6";
{\ar_{\frac{\wp_X}{\pro^{(G)}}} "2";"0"};
{\ar_{\althh} "0";"4"};
{\ar^{\frac{\wp_B}{\pro^{(H)}}} "2";"6"};
{\ar^{\frac{\bfr}{H}} "6";"4"};
{\ar@{<=}_{\kappa} (2.5,-2);(-2.5,2)};
\endxy
\]
as in Proposition \ref{Prop2Pullback}.
Define 1-cells
\[ \frac{m}{\tau}\co\frac{\XY B}{G}\to\frac{F}{G\ti H}\qquad \text{and}\qquad \frac{\ell}{\pro^{(G)}}\co \frac{F}{G\ti H}\to\frac{\XY B}{G} \]
by
\begin{eqnarray*}
&m(x,b)=(x,b,e)&(\fa (x,b)\in\XY B),\\
&\ell(x,b,\eta)=(x,\eta\iv b)&(\fa (x,b,\eta)\in F),\\
&\tau_{(x,b)}(g)=(g,\thh_x(g))&(\fa (x,b)\in\XY B, g\in G),
\end{eqnarray*}
and define a 2-cell $\ep\co\frac{m}{\tau}\ci\frac{\ell}{\pro^{(G)}}\tc\id$ by
\[ \ep_{(x,b,\eta)}=(e,\eta)\qquad(\fa (x,b,\eta)\in F). \]
Then the diagram
\[
\xy
(-45,0)*+{\frac{\XY B}{G}}="2";
(-15,0)*+{\frac{F}{G\ti H}}="4";
(15,0)*+{\frac{\XY B}{G}}="6";
(45,0)*+{\frac{F}{G\ti H}}="8";
{\ar_{\frac{m}{\tau}} "2";"4"};
{\ar_{\frac{\ell}{\pro^{(G)}}} "4";"6"};
{\ar^{\frac{m}{\tau}} "6";"8"};
{\ar@/^1.8pc/^{\id} "2";"6"};
{\ar@/_1.8pc/_{\id} "4";"8"};
{\ar@{=>}^{\ep} (-15,3);(-15,6)};
{\ar@{}|{\circlearrowright} (15,-4);(15,-7)};
\endxy
\]
shows that $\frac{m}{\tau}$ is an equivalence. Thus the diagram
\[
\xy
(26,18)*+{\frac{\XY B}{G}}="-2";
(-10,7)*+{\xg}="0";
(12,7)*+{\frac{F}{G\ti H}}="2";
(-10,-8)*+{\yh}="4";
(12,-8)*+{\frac{B}{H}}="6";
{\ar^{\simeq}_{\frac{m}{\tau}} "-2";"2"};
{\ar@/_1.18pc/_{\frac{p_X}{G}} "-2";"0"};
{\ar@/^1.24pc/^{\althhd} "-2";"6"};
{\ar^(0.56){\frac{\wp_X}{\pro^{(G)}}} "2";"0"};
{\ar_{\althh} "0";"4"};
{\ar^{\frac{\wp_B}{\pro^{(H)}}} "2";"6"};
{\ar^{\frac{\bfr}{H}} "6";"4"};
{\ar@{<=}^{\kappa} (3,-2.5);(-1,0.5)};
{\ar@{}^{_{\circlearrowright}} (12,3);(26,3)};
{\ar@{}_{_{\circlearrowright}} (2,9);(10,20)};
\endxy
\]
and the equality $\kappa\ci\frac{m}{\tau}=\id$ shows that $(\ref{PBEquiv})$ is a bipullback.
\end{proof}

The following immediately follows from the universal property of bipullbacks.
\begin{cor}\label{CorAst}
For any 1-cell $\oc$ in $\Sbb$, put $\Bbbb(\althh)=\ata$.
\begin{enumerate}
\item If $\ep\co\althh\tc\althhp$ is a 2-cell in $\Sbb$, then there is a natural transformation
\[ \Bbbb(\ep)\co \Bbbb(\althh)\tc \Bbbb(\althhp), \]
which is in fact an isomorphism because of the invertibility of $\ep$. Explicitly, this is given by
\[ \Bbbb(\ep)_{\Bb}\co\{(x,b)\mid\al(x)=\bfr(b)\}\to\{(x,b)\mid\al\ppr(x)=\bfr(b)\}\ ; \ (x,b)\mapsto (x,\ep_x\iv b) \]
for each $\Bb\in\Ob(\Hs/Y)$.
\item If $\xg\ov{\althh}{\lra}\yh\ov{\bet}{\lra}\zk$ is a sequence of 1-cells in $\Sbb$, then there is a naturally defined isomorphism
\[ \Bbbb(\bet\ci\althh)\cong\Bbbb(\althh)\ci \Bbbb(\bet). \]
\item For any 0-cell $\xg$ in $\Sbb$, there is a natural isomorphism
\[ \Bbbb(\id_{\xg})\cong\Id_{\Bbbb(\xg)}. \]
\end{enumerate}
These are given in a natural way, resulting in a bifunctor $\Bbbb\co\Sbb^{\op}\to\mathrm{CAT}$.
\end{cor}

\section{Left adjoint of $\ata$}

We construct a left adjoint of $\ata$, for any 1-cell $\althh$ in $\Sbb$.
\begin{rem}
If $\frac{\al}{G}\co\xg\to\yg$ is a $G$-equivariant 1-cell, then a left adjoint of $(\frac{\al}{G})\uas=\al\uas\co\Gs/Y\to\Gs/X$ is given by the composition functor
\[ \al_+=\al\ci-\co\Gs/X\to\Gs/Y. \]
\end{rem}

\begin{dfn}\label{DefSS}
Let $\oc$ be any 1-cell. Define a functor $S_{\althh}\co\Gs/X\to\Hs/\SIm(\althh)$ as follows.

From any morphism $f\co\Aa\to\Aap$ in $\Gs/X$, we obtain a diagram in $\Sbb$
\[
\xy
(-22,20)*+{\frac{A}{G}}="0";
(-22,-20)*+{\frac{A\ppr}{G}}="2";
(-2,8)*+{\frac{\SIm(\althh\ci\frac{\afr}{G})}{H}}="4";
(4,18)*+{}="5";
(-2,-8)*+{\frac{\SIm(\althh\ci\frac{\afr\ppr}{G})}{H}}="6";
(4,-18)*+{}="7";
(26,0)*+{\frac{\SIm(\althh)}{H}}="8";
(-8,0)*+{}="9";
(48,0)*+{\yh}="10";
(5.7,20)*+{}="11";
(5.7,-20)*+{}="12";
(28,12)*+{}="21";
(28,-12)*+{}="22";
{\ar^{\frac{\ups_{(\al\ci\afr)}}{\thh\ppr}} "0";"4"};
{\ar_{\frac{\ups_{(\al\ci\afr\ppr)}}{\thh\pprr}} "2";"6"};
{\ar_{\frac{f}{G}} "0";"2"};
{\ar^{\frac{\ovl{\afr}}{H}} "4";"8"};
{\ar_{\frac{\ovl{\afr}\ppr}{H}} "6";"8"};
{\ar_{\frac{\ovl{f}}{H}} "4";"6"};
{\ar^{\frac{\wt{\al}}{H}} "8";"10"};
{\ar@/^1.8pc/^{\ulthh\ci\frac{\afr}{G}} "0";"8"};
{\ar@/_1.8pc/_{\ulthh\ci\frac{\afr\ppr}{G}} "2";"8"};
{\ar@/^1.8pc/|!{"8";"11"}\hole^{\frac{\wt{(\al\ci\afr)}}{H}} "4";"10"};
{\ar@/_1.8pc/|!{"8";"12"}\hole_{\frac{\wt{(\al\ci\afr\ppr)}}{H}} "6";"10"};
{\ar@{}|\circlearrowright (-17,0);(-13,0)};
{\ar@{}|\circlearrowright "4";"5"};
{\ar@{}|\circlearrowright "6";"7"};
{\ar@{}|\circlearrowright "8";"9"};
{\ar@{}|\circlearrowright "8";"21"};
{\ar@{}|\circlearrowright "8";"22"};
\endxy
\]
where
\[ \althh\ci\frac{\afr}{G}=\frac{\wt{(\al\ci\afr)}}{H}\ci\frac{\ups_{(\al\ci\afr)}}{\thh\ppr}\quad\text{and}\quad \althh\ci\frac{\afr\ppr}{G}=\frac{\wt{(\al\ci\afr\ppr)}}{H}\ci\frac{\ups_{(\al\ci\afr\ppr)}}{\thh\pprr} \]
are $\SIm$-factorizations, and $\ovl{\afr},\ovl{\afr}\ppr,\ovl{f}$ are given by
\begin{eqnarray*}
&\ovl{\afr}\co\HG A\to\HG X\ ;\ [\eta,a]\mapsto [\eta,\afr(a)],&\\
&\ovl{\afr}\ppr\co\HG A\ppr\to\HG X\ ;\ [\eta,a\ppr]\mapsto [\eta,\afr\ppr(a\ppr)],&\\
&\ovl{f}\co\HG A\to\HG A\ppr\ ;\ [\eta,a]\mapsto [\eta,f(a)].&
\end{eqnarray*}
Using this, we define $S_{\althh}$ by
\[ S_{\althh}\Aa=(\SIm(\althh\ci\frac{\afr}{G})\ov{\ovl{\afr}}{\lra}\SIm(\althh)),\qquad S_{\althh}(f)=\ovl{f} \]
for any object $\Aa$ and morphism $f$ in $\Gs/X$.
\end{dfn}

\begin{dfn}\label{DefPlus}
For any 1-cell $\oc$, let $\althh=\frac{\wt{\al}}{H}\ci\ulthh$ be its $\SIm$-factorization, and define a functor $\atp\co\Gs/X\to\Hs/Y$ to be the composition of
\[ \Gs/X\ov{S_{\althh}}{\lra}\Hs/\SIm(\althh)\ov{\wt{\al}_+}{\lra}\Hs/Y. \]
\end{dfn}

\begin{rem}\label{RemPlus}
The following holds.
\begin{enumerate}
\item The functor $S_{\althh}$ depends only on the acting part $\thh$. In particular, for any 1-cell $\althh$, we have $S_{\althh}=S_{\ulthh}\cong(\ulthh)_+$.
\item If $\althh$ is stab-surjective, then $\wt{\al}$ is an isomorphism, and thus $\wt{\al}_+$ is an isomorphism of categories.
\item For an equivariant 1-cell $\frac{\al}{G}$, we have $(\frac{\al}{G})_+\cong\wt{\al}_+\cong\al_+$.
\end{enumerate}
By {\rm (1)}, we denote $S_{\althh}$ simply by $S_{\thh}$, when there is no confusion.
\end{rem}

\begin{prop}\label{PropLeftAdj}
For any 1-cell $\oc$, functor $\atp$ is left adjoint to $\ata$.
\end{prop}
\begin{proof}
%
It suffices to construct a natural bijection
\[ \Phi\co\Hs/Y(\SIm(\althh\ci\frac{\afr}{G})\ov{\wt{(\al\ci\afr)}}{\lra}Y,B\ov{\bfr}{\lra}Y)\ov{\cong}{\longleftrightarrow}\Gs/X(A\ov{\afr}{\lra}X,\XY B\ov{p_X}{\lra}X)\co\Psi \]
for any $\Aa\in\Ob(\Gs/X)$ and $\Bb\in\Ob(\Hs/Y)$.

\bigskip

\noindent\underline{Construction of $\Phi$}

Let $\psi\in\Hs/Y(\SIm(\althh\ci\frac{\afr}{G})\ov{\wt{(\al\ci\afr)}}{\lra}Y,B\ov{\bfr}{\lra}Y)$ be any morphism. Define $\Phi(\psi)$ by
\[ \Phi(\psi)(a)=(\afr(a),\psi([e,a]))\qquad(\fa a\in A). \]
Then $G$-equivariance of $\Phi(\psi)$ follows from
\begin{eqnarray*}
\Phi(\psi)(ga)&=&(\afr(ga),\psi([e,ga]))\\
&=&(g\afr(a),\psi(\thh_{\afr(a)}(g)[e,a]))\\
&=&g\,\Phi(\psi)(a)\qquad(\fa g\in G, a\in A).
\end{eqnarray*}
The equality $p_X\ci\Phi(\psi)=\afr$ is trivially satisfied.


\bigskip

\noindent\underline{Construction of $\Psi$}

Let $\vp\in\Gs/X(A\ov{\afr}{\lra}X,\XY B\ov{p_X}{\lra}X)$ be any morphism. Define $\Psi(\vp)$ by
\[ \Psi(\vp)([\eta,a])=\eta\, p_B\ci\vp(a)\qquad(\fa [\eta,a]\in\SIm(\althh\ci\frac{\afr}{G})), \]
where $p_B\co\XY B\to B$ is the projection($=\al\di$). For any $g\in G$, we have
\[ \Psi(\vp)([\eta\thh_{\afr(a)}(g)\iv,ga])=\eta\thh_{\afr(a)}(g)\iv p_B\ci\vp(ga)=\eta\, p_B\ci\vp(a)=\Psi(\vp)([\eta,a]), \]
which shows the well-definedness of $\Psi(\vp)$.

$H$-equivariance of $\Psi(\vp)$ is obvious. The equality $\bfr\ci\Psi(\vp)=\wt{(\al\ci\afr)}$ follows from
\begin{eqnarray*}
(\bfr\ci\Psi(\vp))([\eta,a])&=&\bfr(\eta\, p_B\ci\vp(a))\ =\ \eta\,\bfr(p_B\ci\vp(a))\\
&=&\eta\, \al(p_X\ci\vp(a))\ =\ \eta\, \al(\afr(a))\\
&=& \wt{(\al\ci\afr)}([\eta,a])\qquad(\fa [\eta,a]\in \SIm ({\althh\ci\frac{\afr}{G}})).
\end{eqnarray*}


\bigskip

\noindent\underline{Confirmation of $\Psi\ci\Phi=\id$}

For any $\psi\in\Hs/Y(\SIm(\althh\ci\frac{\afr}{G})\ov{\wt{(\al\ci\afr)}}{\lra}Y,B\ov{\bfr}{\lra}Y)$, we have
\begin{eqnarray*}
\Psi(\Phi(\psi))([\eta,a])=\eta\,(p_B\ci\Phi(\psi))(a)=\eta\,\psi([e,a])=\psi([\eta,a])
\end{eqnarray*}
for any $[\eta,a]\in\SIm(\althh\ci\frac{\afr}{G})$, and thus $\Psi(\Phi(\psi))=\psi$.


\bigskip

\noindent\underline{Confirmation of $\Phi\ci\Psi=\id$}

For any $\vp\in\Gs/X(A\ov{\afr}{\lra}X,\XY B\ov{p_X}{\lra}X)$, we have
\begin{eqnarray*}
\Phi(\Psi(\vp))(a)=(\afr(a),\Psi(\vp)([e,a]))=(\afr(a),p_B\ci\vp(a))=\vp(a)
\end{eqnarray*}
for any $a\in A$, and thus $\Phi(\Psi(\vp))=\vp$.
\end{proof}

\begin{cor}\label{CorFull}
If $\oc$ is stab-surjective, then the functor $\ata\co\Hs/Y\to\Gs/X$ is fully faithful. Equivalently, the counit $\Lambda\co\atp\ata\tc\Id$ associated to $\atp\lt\ata$ is an isomorphism.
\end{cor}
\begin{proof}
By Proposition \ref{PropAst}, for any $\Bb\in\Ob(\Gs/Y)$, the diagram $(\ref{PBEquiv})$ is a bipullback. As in Proposition \ref{PropStabsurj}, stab-surjectivity of $\althh$ implies that $\althhd$ is also stab-surjective.
Thus diagram $(\ref{PBEquiv})$ gives a factorization of $\althh\ci\frac{p_X}{G}$ into stab-surjective $\althhd$ and equivariant $\frac{\bfr}{H}$. By Remark \ref{PropSImFactorSystem}, we have an $H$-isomorphism
\[ \om\co\SIm(\althh\ci\frac{p_X}{G})\ov{\cong}{\lra}B\ ;\ [\eta,(x,b)]\mapsto\eta\al\di(x,b)=\eta b, \]
compatible with $\wt{(\al\ci p_X)}$ and $\bfr$. The construction of $\Psi$ given in the proof of Proposition \ref{PropLeftAdj} shows
\[ \om=\Psi(\id_{\ata\Bb})=\Lambda_{\Bb}. \]
Thus $\Lambda$ is isomorphic at any $\Bb\in\Ob(\Hs/Y)$.
\end{proof}


\section{Right adjoint of $\ata$}

We construct a right adjoint of $\ata$, for any 1-cell $\althh$ in $\Sbb$.
\begin{rem}
If $\frac{\al}{G}\co\xg\to\yg$ is a $G$-equivariant 1-cell, then a right adjoint functor of $(\frac{\al}{G})\uas=\al\uas$ is given by the functor
\[ \al\bu=\Pi_{\al}\co\Gs/X\to\Gs/Y \]
defined as follows (\cite[section 1]{Tambara}).
\begin{enumerate}
\item For an object $\Aa\in\Ob(\Gs/X)$, define $\Pi_{\al}\Aa=(\Pi_{\al}(A)\ov{\pi_Y}{\lra}Y)$ as follows.
\begin{itemize}
\item[{\rm (i)}] $\Pi_{\al}(A)$ is the set of pairs $(y,\sig)$ of $y\in Y$ and a map $\sig\co\al\iv(y)\to A$ which makes the following diagram commutative.
\[
\xy
(-10,6)*+{\al\iv(y)}="0";
(10,6)*+{A}="2";
(0,-8)*+{X}="4";
(0,10)*+{}="5";
{\ar^(0.54){\sig} "0";"2"};
{\ar@{_(->} "0";"4"};
{\ar^(0.4){\afr} "2";"4"};
{\ar@{}|\circlearrowright "4";"5"};
\endxy
\]
\item[{\rm (ii)}] $G$-action on $\Pi_{\al}(A)$ is given by
\[ g(y,\sig)=(gy,{}^g\sig)\qquad(\fa g\in G, (y,\sig)\in\Pi_{\al}(A)), \]
where ${}^g\sig$ is defined by
\[ {}^g\sig\co\al\iv(gy)\to A\ ;\ x\mapsto g\sig(g\iv x).  \]
\item[{\rm (iii)}] $\pi_Y\co\Pi_{\al}(A)\to Y$ is the projection onto $Y$, which is obviously $G$-equivariant.
\end{itemize}
\item For a morphism $f\co\Aa\to\Aap$ in $\Gs/X$, define $\Pi_{\al}(f)\co\Pi_{\al}\Aa\to\Pi_{\al}\Aap$ by
\[ \Pi_{\al}(f)\co\Pi_{\al}(A)\to\Pi_{\al}(A\ppr)\ ;\ (y,\sig)\mapsto(y,f\ci\sig). \]
\end{enumerate}
\end{rem}


\begin{dfn}\label{DefEqual}
Let $\oc$ be a 1-cell in $\Sbb$, and let $\Aa=\Aar$ be an object in $\Gs/X$. At a point $a\in A$, we consider the following equivalence relation $\un{(\althh,\afr)}{\eq}$ on $G$.
\begin{itemize}
\item[-] Two elements $g,g\ppr\in G$ are said to satisfy $g\un{(\althh,\afr)}{\eq}g\ppr$ at $a$, if they satisfy $\thh_{\afr(a)}(g)=\thh_{\afr(a)}(g\ppr)$ and $g\afr(a)=g\ppr\afr(a)$.
\end{itemize}
\end{dfn}

\begin{rem}\label{RemEqual}
Let $\althh,\Aa$ and $a\in A$ be as above. For any $g,g\ppr\in G$, the following holds.
\begin{enumerate}
\item This equivalence relation depends only on the acting part $\thh$. In particular we have
\[ g\un{(\althh,\afr)}{\eq}g\ppr\ \text{at}\ a\quad\LR\quad g\un{(\ulthh,\afr)}{\eq}g\ppr\ \text{at}\ a. \]
\item For any $g_0\in G$, we have
\[ g\un{(\althh,\afr)}{\eq}g\ppr\ \text{at}\ g_0a\quad\LR\quad gg_0\un{(\althh,\afr)}{\eq}g\ppr g_0\ \text{at}\ a. \]
\item For an equivariant 1-cell $\frac{\al}{G}$,
\[ g\un{(\frac{\al}{G},\afr)}{\eq}g\ppr\ \text{at}\ a\quad\LR\quad g=g\ppr \]
holds for any $a\in A$.
\end{enumerate}
By {\rm (1)}, we denote $\un{(\althh,\afr)}{\eq}$ simply by $\un{(\thh,\afr)}{\eq}$ when there is no confusion.

\end{rem}


\begin{dfn}\label{DefFix}
Let $\oc$ be any 1-cell in $\Sbb$. Define an endofunctor
\[ (-)^{\althh}\co\Gs/X\to\Gs/X \]
as follows.
\begin{enumerate}
\item For an object $\Aa$ in $\Gs/X$, define $\Aa\ual=(A\ual\ov{\afr\ci\iota}{\lra}X)\in\Ob(\Gs/X)$ by the following.
\begin{itemize}
\item[{\rm (i)}] $A\ual\se A$ is the subset, consisting of $a\in A$ satisfying
\[ g\un{(\thh,\afr)}{\eq}g\ppr\ \text{at}\ a\quad \Longrightarrow\quad ga=g\ppr a \]
for any $g,g\ppr\in G$. By Remark \ref{RemEqual} {\rm (2)}, this becomes a $G$-subset of $A$.
\item[{\rm (ii)}] $\iota\co A\ual\hookrightarrow A$ denotes the inclusion.
\end{itemize}
\item For a morphism $f\co\Aa\to\Aap$ in $\Gs/X$, since we have $f(A\ual)\se A^{\prime\althh}$, we define $f\ual\co\Aa\ual\to\Aap\ual$ by $f\ual=f|_{A\ual}$.
\end{enumerate}
\end{dfn}

\begin{rem}\label{RemFix}
Let $\oc$ be any 1-cell in $\Sbb$. For any $\Aa\in\Ob(\Gs/X)$, applying $(-)\ual$ to the unique morphism
\[ \afr\co\Aa\to(X,\id_X), \]
we obtain the following commutative diagram
\[
\xy
(-10,6)*+{A\ual}="0";
(10,6)*+{X\ual}="2";
(17.3,5.8)*+{=X}="6";
(0,-8)*+{X}="4";
(0,10)*+{}="5";
{\ar^{\afr\ual} "0";"2"};
{\ar_(0.4){\afr\ci\iota} "0";"4"};
{\ar^(0.4){\id_X} "2";"4"};
{\ar@{}|\circlearrowright "4";"5"};
\endxy
\]
which means $\afr\ci\iota=\afr\ual$. For this reason, we abbreviately write $\Aa\ual=(A\ual,\afr\ual)$ in the rest. Moreover, since the functor $(-)\ual$ depends only on the acting part $\thh$ by Remark \ref{RemEqual}, we denote it simply by $(-)\uth$ if there is no confusion. Thus we often write as $\Aa\uth=(A\uth,\afr\uth)$ for an object $\Aa\in\Ob(\Gs/X)$. 
For any 1-cell $\althh$, the functor $(-)\uth$ is idempotent. Namely, we have $(-)\uth\ci(-)\uth=(-)\uth$.
\end{rem}


\begin{dfn}\label{DefBullet}
For any 1-cell $\oc$, take its $\SIm$-factorization $\althh=\frac{\wt{\al}}{H}\ci\ulthh$, and define a functor $\atb\co\Gs/X\to\Hs/Y$ to be the composition of the following sequence of functors.
\[ \Gs/X\ov{(-)\uth}{\lra}\Gs/X\ov{S_{\thh}}{\lra}\Hs/\SIm(\althh)\ov{\Pi_{\wt{\al}}}{\lra}\Hs/Y \]
\end{dfn}

\begin{rem}\label{RemBullet}
By Remarks \ref{RemPlus} and \ref{RemEqual}, we have $(\frac{\al}{G})\bu\cong\wt{\al}\bu\cong\al\bu$ for an equivariant 1-cell $\frac{\al}{G}$.
\end{rem}


\begin{lem}\label{LemRightAdj1}
Let $\oc$ be any 1-cell. If we denote the essential image of the functor
\[ (-)\uth\co\Gs/X\to\Gs/X \]
(namely, the full subcategory of $\Gs/X$ consisting of objects isomorphic to objects of the form $\Aa\uth$ for some $\Aa\in\Ob(\Gs/X)$) by $\Fcal\se\Gs/X$, then the following holds.
\begin{enumerate}
\item An object $\Aa\in\Ob(\Gs/X)$ belongs to $\Fcal$ if and only if the subset $A\uth\se A$ is equal to $A$ itself.
\item $(-)\uth\co\Gs/X\to\Fcal$ is right adjoint to the inclusion $\Fcal\hookrightarrow\Gs/X$.
\end{enumerate}
\end{lem}
\begin{proof}
{\rm (1)} follows from the idempotency of $(-)\uth$.

{\rm (2)} Let $\Aa\in\Ob(\Gs/X)$ and $\Aap\in\Ob(\Fcal)$ be any pair of objects. A natural bijection
\[ \Gs/X(\Aap,\Aa)\ov{\cong}{\longleftrightarrow}\Fcal(\Aap,\Aa\uth) \]
exists, since we have $f(A\ppr)=f(A^{\prime\thh})\se A\uth$ for any $f\in\Gs/X(\Aap,\Aa)$.
\end{proof}


\begin{lem}\label{LemRightAdj2}
Let $\oc$ be any 1-cell, and let $\Fcal$ be as in Lemma \ref{LemRightAdj1}.
\begin{enumerate}
\item The functor $\ata\co\Hs/Y\to\Gs/X$ factors through $\Fcal$. Namely, we have $\ata\Bb\in\Ob(\Fcal)$ for any $\Bb\in\Ob(\Hs/Y)$.
\item Assume $\althh$ is stab-surjective. Then
\[ \atp|_{\Fcal}\co\Fcal\to\Hs/Y\qquad \text{and}\qquad\ata\co\Hs/Y\to\Fcal \]
are mutually quasi-inverses.
\end{enumerate}
\end{lem}
\begin{proof}
{\rm (1)} By definition, we have $\ata\Bb=(\XY B\ov{p_X}{\lra}X)$. Let $(x,b)\in\XY B$ be any point. For any $g,g\ppr\in G$ satisfying $g\un{(\thh,p_X)}{\eq}g\ppr$ at $(x,b)$, the equalities
\[ \thh_x(g)=\thh_x(g\ppr)\quad\text{and}\quad gx=g\ppr x \]
imply
\[ g(x,b)=(gx,\thh_x(g)b)=(g\ppr x,\thh_x(g\ppr)b)=g\ppr(x,b). \]
This means $(x,b)\in(\XY B)\uth$.

{\rm (2)} 
By Corollary \ref{CorFull}, we have $\atp\ata\ov{\cong}{\Longrightarrow}\Id$.
So it remains to show that the unit $\Id\tc\ata\atp$ gives an isomorphism at any $\Aa\in\Ob(\Fcal)$. By definition, we have
\begin{eqnarray*}
\ata\atp\Aa&\cong&\ata\atp(\Aa\uth)\\
&=&\ata(\HG A\uth\ov{\wt{(\al\ci\afr\uth)}}{\lra}Y)%
\ =\ (\XY(\HG A\uth)\ov{p_X}{\lra}X)
\end{eqnarray*}
with
\[ \XY(\HG A\uth)=\{ (x,[\eta,a])\in X\ti(\HG A\uth)\mid \al(x)=\eta\,\al(\afr(a)) \}. \]
By the construction of $\Phi$ in the proof of Proposition \ref{PropLeftAdj}, the unit is given by the $G$-map
\[ \lam\co A\uth\to \XY(\HG A\uth)\ ;\ a\mapsto(\afr(a),[e,a]) \]
at $\Aa\uth$. It suffices to give its inverse map. For any $(x,[\eta,a])\in \XY(\HG A\uth)$, stab-surjectivity of $\althh$ and the condition $\al(x)=\eta\,\al(\afr(a))$ imply the existence of an element $g\in G$ satisfying
\begin{equation}\label{EqS1}
x=g\afr(a)\quad\text{and}\quad\eta=\thh_{\afr(a)}(g).
\end{equation}
If we put $\varpi(x,[\eta,a])=ga$, this gives a well-defined map
\[ \varpi\co\XY(\HG A\uth)\to A\uth. \]
In fact, if $[\eta,a]=[\eta\ppr,a\ppr]$ holds in $\HG A\uth$, then by definition there is $g_0\in G$ satisfying
\begin{equation}\label{EqS2}
\eta=\eta\ppr\thh_{\afr(a)}(g_0)\quad\text{and}\quad a\ppr=g_0a.
\end{equation}
If $g\ppr\in G$ satisfies
\begin{equation}\label{EqS3}
x=g\ppr\afr(a\ppr)\quad\text{and}\quad\eta\ppr=\thh_{\afr(a\ppr)}(g\ppr),
\end{equation}
then $(\ref{EqS1}),(\ref{EqS2}),(\ref{EqS3})$ imply $g\un{(\thh,\afr)}{\eq}g\ppr g_0$ at $a$. Since $a\in A\uth$, this means
\[ ga=g\ppr g_0a=g\ppr a\ppr. \]
It can be easily verified that $\varpi$ is inverse to $\lam$.
\end{proof}


\begin{prop}\label{PropRightAdj}
For any 1-cell $\oc$, functor $\atb$ is right adjoint to $\ata$.
\end{prop}
\begin{proof}
Remark that we have natural isomorphisms
\[ \ata\cong(\ulthh)\uas\ci\wt{\al}\uas,\quad \atb\cong\wt{\al}\bu\ci(\ulthh)\bu \]
and the adjointness $\wt{\al}\uas\lt\wt{\al}\bu$. Thus replacing $\ulthh$ by $\althh$, it suffices to show the adjointness $\ata\lt\atb$ for stab-surjective 1-cell $\althh$.

Assume $\althh$ is stab-surjective. In this case, since $\wt{\al}\uas$ is an equivalence of categories, it follows $\wt{\al}_+\cong\wt{\al}\bu$. If we denote the inclusion $\Fcal\hookrightarrow \Gs/X$ by $j$, then by Lemmas \ref{LemRightAdj1} and \ref{LemRightAdj2}, we have the following sequence of adjoint pairs.
\[
\xy
(-28,0)*+{\Hs/Y}="0";
(0,14)*+{}="1";
(0,0)*+{\Fcal}="2";
(28,0)*+{\Gs/X}="4";
{\ar@/^2.8pc/^{\ata} "0";"4"};
{\ar@/^1.2pc/^(0.46){} "0";"2"};
{\ar@/^1.2pc/^{\wt{\al}_+\ci S_{\thh}\cong\wt{\al}\bu\ci S_{\thh}} "2";"0"};
{\ar@/^1.2pc/^(0.46){j} "2";"4"};
{\ar@/^1.2pc/^{(-)\uth} "4";"2"};
{\ar@{}|{\perp} (-14,2);(-14,-2)};
{\ar@{}|{\perp} (14,2);(14,-2)};
{\ar@{}|\circlearrowright "1";"2"};
\endxy
\]
It follows $\ata\lt\wt{\al}\bu\ci S_{\thh}\ci(-)\uth=\atb$.
\end{proof}


\section{Six operations}

So far, we have associated an adjoint triplet
\begin{equation}\label{DiagAdjTriplet}
\xy
(-16,0)*+{\Hs/Y}="0";
(16,0)*+{\Gs/X}="2";
{\ar@/_1.8pc/_{\atp} "2";"0"};
{\ar|*+{_{\ata}} "0";"2"};
{\ar@/^1.8pc/^{\atb} "2";"0"};
{\ar@{}|{\perp} (0,6);(0,3)};
{\ar@{}|{\perp} (0,-3);(0,-6)};
\endxy
\end{equation}
to any 1-cell $\oc$. By the decomposition into orbits and $\SIm$-factorizations, we see any 1-cell in $\Sbb$ can be constructed from the following two typical types of morphisms (up to equivalences), by using unions and compositions.
\begin{enumerate}
\item[{\rm [I]}] For a subgroup $H\le G$, the unique map $p\co G/H\to G/G=\pt$ induces a $G$-equivariant 1-cell
\begin{equation}\label{MorType1}
\frac{p}{G}\co\frac{G/H}{G}\to\ptg.
\end{equation}
If we compose this with the $\Ind$-equivalence $\pth\ov{\simeq}{\lra}\frac{G/H}{G}$, we obtain $\frac{\pt}{\iota}\co\pth\to\ptg$, where $\iota\co H\hookrightarrow G$ is the inclusion.
\item[{\rm [II]}] For a normal subgroup $N\nm G$, the quotient homomorphism $q\co G\to G/N$ induces a stab-surjective 1-cell
\begin{equation}\label{MorType2}
\ptq\co\ptg\to\frac{\pt}{G/N}.
\end{equation}
\end{enumerate}

\begin{ex}\label{Ex5.1}
Let $\althh$ be a 1-cell. For types {\rm [I]} and {\rm [II]}, each of functors $\atp,\ata$ and $\atb$ is isomorphic to the following functors.
\begin{center}
\begin{tabular}
[c]{|c|c|c|}\hline
& type {\rm [I]} & type {\rm [II]}\\\hline
$\atp$ & $\Ind^G_H$ & $\Orb^G_{G/N}$ \\\hline
$\ata$ & $\Res^G_H$ & $\Inf^G_{G/N}$ \\\hline
$\atb$ & $\Jnd^G_H$ & $\Inv^G_{G/N}$ \\\hline
\end{tabular}
\end{center}
Thus $(\ref{DiagAdjTriplet})$ recovers the following six operations (cf. \cite{Yoshida}) for finite groups.
\[
\xy
(-26,0)*+{\Hs}="0";
(0,0)*+{\Gs}="2";
(26,0)*+{{}_{G/N}\sett}="4";
{\ar@{<-}@/_1.6pc/_{\Ind} "2";"0"};
{\ar@{<-}|*+{_{\Res}} "0";"2"};
{\ar@{<-}@/^1.6pc/^{\Jnd} "2";"0"};
{\ar@{<-}@/_1.6pc/_{\Orb} "4";"2"};
{\ar@{<-}|*+{_{\Inf}} "2";"4"};
{\ar@{<-}@/^1.6pc/^{\Inv} "4";"2"};
{\ar@{}|{\perp} (-13,5);(-13,3)};
{\ar@{}|{\perp} (-13,-3);(-13,-5)};
{\ar@{}|{\perp} (13,5);(13,3)};
{\ar@{}|{\perp} (13,-3);(13,-5)};
\endxy
\]
\end{ex}
\begin{proof}
For type {\rm [I]}, this is well-known. (As for $\Jnd$, see for example \cite{Tambara}.) Since it can be also easily verified for $\Inf$, we only show for $\Orb$ and $\Inv$.
Let $\althh=\ptq$ be as in $(\ref{MorType2})$. In this case, we have $\SIm(\ptq)=\pt$, and canonically identify as 
\[ \Gs/\pt=\Gs,\quad  \Hs/\pt=\Hs. \]
We may assume $(\ptq)_+\cong S_{\ptq}$ and $(\ptq)\bu\cong S_{\ptq}\ci(-)^{\ptq}$.

\bigskip
\noindent\underline{Functor $\Orb$.}

By Remark \ref{RemPlus}, the functor $(\ptq)_+$ sends an object $A\in\Ob(\Gs)$ to
\[ \SIm(\frac{A}{G}\ov{\ptq}{\lra}\frac{\pt}{G/N})=(G/N)\un{G}{\ti}A=((G/N)\ti A)/\sim, \]
where $(\xi N,a)$ and $(\xi\ppr N,a\ppr)$ in $(G/N)\ti A$ are equivalent if and only if there exists $g\in G$ satisfying
\[ \xi gN=\xi\ppr N\quad\text{and}\quad a\ppr=ga. \]
Thus we have a $G/N$-isomorphism
\[ A/N\ov{\cong}{\lra}((G/N)\ti A)/\sim\ ;\ \overline{a}\mapsto [e,a] \]
which gives a natural isomorphism $\Orb^G_{G/N}\ov{\cong}{\Longrightarrow}(\ptq)_+$.

\bigskip
\noindent\underline{Functor $\Inv$.}

For any $A\in\Ob(\Gs)$, we have
\begin{eqnarray*}
A^{\ptq}&=&\{ a\in A\mid gN=g\ppr N\ \Longrightarrow\ ga=g\ppr a\quad(\fa g,g\ppr\in G)\}\\
&=&\{ a\in A\mid na=a\quad(\fa n\in N)\}\ =\ A^N.
\end{eqnarray*}
Applying $S_{\ptq}$ to $A^N$ simply means that we regard $A^N$ as a $G/N$-set. This induces a natural isomorphism $\Inv^G_{G/N}\ov{\cong}{\Longrightarrow}S_{\ptq}\ci (-)^{\ptq}$.
\end{proof}


\section{Derivator-like properties and associated semi-Mackey functors}

As in \cite{Groth}, a derivator is a strict 2-functor $\Dbb\co\mathrm{Cat}^{\op}\rightarrow\mathrm{CAT}$ satisfying conditions {\rm (Der1),$\ldots$,(Der4)}, from the 2-category of small categories $\mathrm{Cat}$. (For the detail, see \cite{Groth}.)
In this section, we show that the bifunctor $\Bbbb\co\Sbb^{\op}\to\mathrm{CAT}$ satisfies properties analogous to them, when $\mathrm{Cat}$ is replaced by $\Sbb$.


In the following, we recall conditions {\rm (Der1),$\ldots$,(Der4)} from \cite{Groth} one by one, and consider their analogs. 
Let $\Dbb\co\mathrm{Cat}^{\op}\rightarrow\mathrm{CAT}$ be a strict 2-functor (i.e. a prederivator (\cite[Definition 1.1]{Groth})). For any 1-cell $u\co J\to K$ in $\mathrm{Cat}$, denote $\Dbb(u)$ by $u\uas$. Condition {\rm (Der 1)} is as follows.
\begin{condition}({\rm (Der 1)} in \cite[Definition 1.5]{Groth}.)
$\Dbb$ sends coproducts to products.
Namely, for any pair of 0-cells $J_1,J_2$ in $\mathrm{Cat}$, if we take their coproduct
\[ J_1\ov{u_1}{\lra}J_1\am J_2\ov{u_2}{\lla}J_2, \]
then 
\[ (u_1\uas,u_2\uas)\co \Dbb(J_1\am J_2)\to \Dbb(J_1)\ti\Dbb(J_2) \]is an equivalence of categories.
\end{condition}

Its analog for $\Bbbb$ is the following.
\begin{prop}\label{PropDer1}
$\Bbbb$ sends bicoproducts to products.
Namely, for any pair of 0-cells $\xg,\yh$ in $\Sbb$, if we take their bicoproduct
\[ \xg\ov{\althh}{\lra}\xg\am\yh\ov{\bet}{\lla}\yh, \]
then 
\[ ((\althh)\uas,(\bet)\uas)\co \Bbbb(\xg\am\yh)\to \Bbbb(\xg)\ti \Bbbb(\yh) \]
is an equivalence of categories.
\end{prop}
\begin{proof}
As in Proposition \ref{Prop2CoprodVari}, a bicoproduct of $\xg$ and $\yh$ is given by
\[ \xg\ov{\frac{\ups_X}{\iog}}{\lra}\frac{\Ind_{\iog}X\am\Ind_{\ioh}Y}{G\times H}\ov{\frac{\ups_Y}{\ioh}}{\lla}\yh \]
for any pair of 0-cells $\xg,\yh$ in $\Sbb$. Thus the pair of functors $(\frac{\ups_X}{\iog})\uas,(\frac{\ups_Y}{\ioh})\uas$ induces an equivalence of categories
\begin{eqnarray*}
\Bbbb(\frac{\Ind_{\iog}X\am\Ind_{\ioh}Y}{G\times H})&\simeq&(G\ti H)\sett/\big(\Ind_{\iog}X\am\Ind_{\ioh}Y\big)\\
&\simeq&\big((G\ti H)\sett/\Ind_{\iog}X\big)\times\big((G\ti H)\sett/\Ind_{\ioh}Y\big)\\
&\simeq&\big(G\sett/X\big)\times\big(H\sett/Y\big)\\
&=&\Bbbb(\xg)\ti \Bbbb(\yh).
\end{eqnarray*}
\end{proof}

Remark that in $\mathrm{Cat}$, the one-object category $\pt$ is terminal. Moreover, for any 0-cell $K$ in $\mathrm{Cat}$, there is a one-to-one correspondence between objects $k\in\Ob(K)$ and functors $i_k\co\pt\to K$ which sends the unique object in $\pt$ to $k$.
Condition {\rm (Der 2)} can be stated as follows.
\begin{condition}({\rm (Der 2)} in \cite[Definition 1.5]{Groth}.)
Let $K$ be any 0-cell in $\mathrm{Cat}$. For any morphism $f\co A\to A\ppr$ in $\Dbb(K)$, the following are equivalent.
\begin{enumerate}
\item $f$ is an isomorphism in $\Dbb(K)$.
\item $i_k\uas (f)\co i_k\uas (A)\to i_k\uas (A\ppr)$ is an isomorphism in $\Dbb(\pt)$ for any $k\in\Ob(K)$.
\end{enumerate}
\end{condition}

Its analog for $\Bbbb$ is the following.
\begin{prop}\label{PropDer2}
Let $\xg$ be any 0-cell in $\Sbb$. For any morphism $f\co\Aa\to\Aap$ in $\Bbbb(\xg)=\Gs/X$, the following are equivalent.
\begin{enumerate}
\item $f$ is an isomorphism in $\Gs/X$.
\item $f|_{\afr\iv(x)}\co \afr\iv(x)\to\afr^{\prime-1}(x)$ is an isomorphism in $\Bbbb(\pte)=\sett$ $($i.e. a bijection of sets$)$ for any $x\in X$.
\item $f|_{\afr\iv(X_0)}\co \afr\iv(X_0)\to\afr^{\prime-1}(X_0)$ is an isomorphism for any $G$-orbit $X_0\se X$.
\end{enumerate}
\end{prop}
\begin{proof}
This is trivial.
\end{proof}


Condition {\rm (Der 3)} is as follows.
\begin{condition}({\rm (Der 3)} in \cite[Definition 1.5]{Groth}.)
For any 1-cell $u\co J\to K$ in $\mathrm{Cat}$, there is an adjoint triplet $u_!\lt u\uas\lt u\sas$, where $u\uas=\Dbb(u)$.
\end{condition}

Its analog for $\Bbbb$ is the following.
\begin{prop}\label{PropDer3}
For any 1-cell $\oc$, we have an adjoint triplet $\atp\lt\ata\lt\atb$.
\end{prop}
\begin{proof}
This is already shown.
\end{proof}

In \cite{Groth}, condition {\rm (Der4)} is written as follows (\cite[Proposition 1.26]{Groth}). For any pair of functors $J_1\ov{u_1}{\lra}K\ov{u_2}{\lla}J_2$, let $\uu$ be the category defined by the following.
\begin{itemize}
\item[{\rm (i)}] An object in $\uu$ is a triplet $(j_1,j_2,\al)$ of $j_1\in\Ob(J_1),j_2\in\Ob(J_2)$ and $\al\in K(u_1(j_1),u_2(j_2))$.
\item[{\rm (ii)}] A morphism $(f_1,f_2)\co (j_1,j_2,\al)\to (j\ppr_1,j\ppr_2,\al\ppr)$ is a pair of $f_1\in J_1(j_1,j_1\ppr)$ and $f_2\in J_2(j_2,j_2\ppr)$ satisfying $u_2(f_2)\ci\al=\al\ppr\ci u_1(f_1)$. Composition is naturally induced from those in $J_1$ and $J_2$.
\end{itemize}
Projection functors $\pro_1\co\uu\to J_1$ and $\pro_2\co\uu\to J_2$ are naturally defined, which send $(j_1,j_2,\al)\in\Ob\uu$ to $\pro_1(j_1,j_2,\al)=j_1$ and $\pro_2(j_1,j_2,\al)=j_2$. They form a diagram
\begin{equation}\label{DiagComma}
\xy
(-9,7)*+{\uu}="0";
(9,7)*+{J_1}="2";
(-9,-7)*+{J_2}="4";
(9,-7)*+{K}="6";
{\ar^(0.56){\pro_1} "0";"2"};
{\ar_{\pro_2} "0";"4"};
{\ar^{u_1} "2";"6"};
{\ar_{u_2} "4";"6"};
{\ar@{=>}^{\ep} (2,2);(-2,-2)};
\endxy
\end{equation}
where $\ep$ is a natural transformation defined by $\ep_{(j_1,j_2,\al)}=\al$.
Condition equivalent to {\rm (Der 4)} can be stated as follows.
\begin{condition}\label{Cond4}({\rm (Der 4)}; Proposition 1.26 in \cite{Groth}.)
For any $(\ref{DiagComma})$, natural isomorphisms
\begin{equation}
(\pro_2)_{!}\ci(\pro_1)\uas\ov{\cong}{\Longrightarrow}u_2\uas\ci u_{1!}\quad\text{and}\quad u_1\uas\ci u_{2\ast}\ov{\cong}{\Longrightarrow}(\pro_1)\sas\ci(\pro_2)\uas
\end{equation}
are obtained from the adjoint property.
\end{condition}

\begin{rem}\label{RemComma}
For any pair of functors $J_1\ov{u_1}{\lra}K\ov{u_2}{\lla}J_2$ as above, diagram $(\ref{DiagComma})$ satisfies the following universal property for any diagram
\[
\xy
(-9,7)*+{I}="0";
(9,7)*+{J_1}="2";
(-9,-7)*+{J_2}="4";
(9,-7)*+{K}="6";
{\ar^{F_1} "0";"2"};
{\ar_{F_2} "0";"4"};
{\ar^{u_1} "2";"6"};
{\ar_{u_2} "4";"6"};
{\ar@{=>}^{\vp} (2,2);(-2,-2)};
\endxy
\]
in $\mathrm{Cat}$.
\begin{enumerate}
\item We have a functor $\wt{F}\co I\to\uu$ which sends
\begin{itemize}
\item[-] an object $X\in\Ob(I)$ to $\wt{F}(X)=(F_1(X),F_2(X),\vp_X)$,
\item[-] a morphism $i\in I(X,Y)$ to $\wt{F}(i)=(F_1(i),F_2(i))$.
\end{itemize}
This functor makes the following diagram commutative, and satisfies $\vp\bu=\ep\ci\wt{F}$.
\[
\xy
(-22,16)*+{I}="-2";
(-9,6)*+{\uu}="0";
(9,6)*+{J_1}="2";
(-9,-8)*+{J_2}="4";
{\ar^(0.56){\wt{F}} "-2";"0"};
{\ar@/^1.10pc/^{F_1} "-2";"2"};
{\ar@/_1.18pc/_{F_2} "-2";"4"};
{\ar_(0.56){\pro_1} "0";"2"};
{\ar^{\pro_2} "0";"4"};
{\ar@{}^{_{\circlearrowright}} (-12,3);(-18,5)};
{\ar@{}_{_{\circlearrowright}} (-8,7);(-10,14)};
\endxy
\]
\item For any other functor $F\co I\to\uu$ and natural transformations
\[
\xy
(-22,16)*+{I}="-2";
(-9,6)*+{\uu}="0";
(9,6)*+{J_1}="2";
(-9,-8)*+{J_2}="4";
{\ar^(0.56){F} "-2";"0"};
{\ar@/^1.10pc/^{F_1} "-2";"2"};
{\ar@/_1.18pc/_(0.56){F_2} "-2";"4"};
{\ar_(0.56){\pro_1} "0";"2"};
{\ar^{\pro_2} "0";"4"};
{\ar@{=>}_{_{\xi_1}} (-8,9);(-6,13)};
{\ar@{=>}^{_{\xi_2}} (-13.5,3);(-17.5,0.5)};
\endxy
\]
which make the diagram of natural transformations
\[
\xy
(-13,7)*+{u_1\ci\pro_1\ci F}="0";
(12,7)*+{u_1\ci F_1}="2";
(-13,-6)*+{u_2\ci\pro_2\ci F}="4";
(12,-6)*+{u_2\ci F_2}="6";
{\ar@{=>}^(0.6){u_1\ci\xi_1} "0";"2"};
{\ar@{=>}_{\ep\ci F} "0";"4"};
{\ar@{=>}^{\vp} "2";"6"};
{\ar@{=>}_(0.6){u_2\ci\xi_2} "4";"6"};
{\ar@{}|\circlearrowright "0";"6"};
\endxy
\]
commutative, there exists a unique natural transformation $\xi\co F\tc \wt{F}$ which satisfies $\xi_1=\pro_1\ci\xi$ and $\xi_2=\pro_2\ci\xi$. In fact, $\xi$ is given by
\[ \xi_X=(\xi_{1X},\xi_{2X})\qquad(\fa X\in\Ob(I)). \]
\end{enumerate}
\end{rem}
If we replace $\mathrm{Cat}$ by $\Sbb$, this is satisfied by the bipullback obtained in Proposition \ref{Prop2Pullback}, as Remark 3.2.24 in \cite{N_BisetMackey} implies. In this sense, the following proposition can be seen as an analog of {\rm (Der4)}, for $\Bbbb$.

\begin{prop}\label{PropDer4}
For any bipullback
\begin{equation}\label{Diag2PB}
\xy
(-8,7)*+{\wl}="0";
(8,7)*+{\xg}="2";
(-8,-7)*+{\yh}="4";
(8,-7)*+{\zk}="6";
{\ar^{\frac{\gamma}{\mu}} "0";"2"};
{\ar_{\frac{\delta}{\nu}} "0";"4"};
{\ar^{\althh} "2";"6"};
{\ar_{\bet} "4";"6"};
{\ar@{<=}^{\kappa} (-2,-2);(2,2)};
\endxy
\end{equation}
in $\Sbb$, we have the following natural isomorphisms.
\begin{itemize}
\item[{\rm (i)}] $(\bet)\uas\ci\atp\ov{\cong}{\Longrightarrow}(\frac{\delta}{\nu})_+\ci(\frac{\gamma}{\mu})\uas$.
\item[{\rm (ii)}]  $(\bet)\uas\ci\atb\ov{\cong}{\Longrightarrow}(\frac{\delta}{\nu})\bu\ci(\frac{\gamma}{\mu})\uas$.
\end{itemize}
\end{prop}
\begin{rem}
\begin{enumerate}
\item As for diagram $(\ref{DiagComma})$, there is no 2-cell $\ep\ppr$ in general
\[
\xy
(-9,7)*+{\uu}="0";
(9,7)*+{J_1}="2";
(-9,-7)*+{J_2}="4";
(9,-7)*+{K}="6";
{\ar^(0.56){\pro_1} "0";"2"};
{\ar_{\pro_2} "0";"4"};
{\ar^{u_1} "2";"6"};
{\ar_{u_2} "4";"6"};
{\ar@{=>}^{\ep\ppr} (-2,-2);(2,2)};
\endxy,
\]
since 2-cells in $\mathrm{Cat}$ are not invertible. Condition \ref{Cond4} is not symmetric, and does not imply $(\pro_1)_{!}\ci(\pro_2)\uas\cong u_1\uas\ci u_{2!}$ or $u_2\uas\ci u_{1\ast}\cong(\pro_2)\sas\ci(\pro_1)\uas$.

However in our case in $\Sbb$, the transposed diagram of $(\ref{Diag2PB})$ is of course again a bipullback, because of the symmetry in its definition (remark that 2-cells in $\Sbb$ are invertible). Thus we also have isomorphisms
\begin{itemize}
\item[{\rm (i)$\ppr$}] $\ata\ci(\bet)_+\ov{\cong}{\Longrightarrow}(\frac{\gamma}{\mu})_+\ci(\frac{\delta}{\nu})\uas$,
\item[{\rm (ii)$\ppr$}]  $\ata\ci(\bet)\bu\ov{\cong}{\Longrightarrow}(\frac{\gamma}{\mu})\bu\ci(\frac{\delta}{\nu})\uas$.
\end{itemize}
\item As in \cite{Groth}, in the definition of a derivator, these natural isomorphisms are explicitly constructed by using adjointness. However in this article, we do not specify the construction of these isomorphisms, since we do not need it to obtain semi-Mackey functors in the later argument.
\end{enumerate}
\end{rem}
\begin{proof}(Proof of Proposition \ref{PropDer4}.)
{\rm (i)} We may assume $(\ref{Diag2PB})$ is of the form obtained in Proposition \ref{Prop2Pullback}.
We shall show $(\bet)\uas\ci\atp\cong(\frac{\wp_Y}{\pro^{(H)}})_+\ci(\frac{\wp_X}{\pro^{(G)}})\uas$.
For any $\Aa\in\Ob(\Gs/X)$, we have
\begin{eqnarray*}
(\bet)\uas\ci\atp\Aar&=&(\bet)\uas(\SIm(\althh\ci\frac{\afr}{G})\ov{\wt{(\al\ci\afr)}}{\lra}Z)\\
&=&(Y\un{Z}{\ti}\SIm(\althh\ci\frac{\afr}{G})\ov{p_Y}{\lra}Y)
\end{eqnarray*}
and
\begin{eqnarray*}
(\frac{\wp_Y}{\pro^{(H)}})_+\ci(\frac{\wp_X}{\pro^{(G)}})\uas\Aar&=&(\frac{\wp_Y}{\pro^{(H)}})_+(F\un{X}{\ti}A\ov{p_F}{\lra}F)\\
&=&(\SIm(\frac{\wp_Y}{\pro^{(H)}}\ci\frac{p_F}{G\ti H})\ov{\wt{(\wp_Y\ci p_F)}}{\lra}Y).
\end{eqnarray*}
By the definition of $\SIm(\althh\ci\frac{\afr}{G})=K\un{G}{\ti}A$, the set $Y\un{Z}{\ti}\SIm(\althh\ci\frac{\afr}{G})$ is the quotient of
\[ Y\un{Z}{\ti}(K\ti A)=\{(y,k,a)\in Y\ti K\ti A\mid \be(y)=k\al(\afr(a))\} \]
by the relation
\begin{eqnarray*}
&&(y,k,a)\sim (y\ppr,k\ppr,a\ppr)\\
&&\Longleftrightarrow\ y=y\ppr\ \text{and, there exists}\ g\in G\ \text{satisfying}\ k=k\ppr\thh_{\afr(a)}(g),\ a\ppr=ga.
\end{eqnarray*}
$H$-action on $Y\un{Z}{\ti}\SIm(\althh\ci\frac{\afr}{G})$ is given by
\[ h(y,[k,a])=(hy,[\tau_y(h)k,a])\quad(\fa h\in H, (y,[k,a])\in Y\un{Z}{\ti}\SIm(\althh\ci\frac{\afr}{G})). \]


On the other hand, $\SIm(\frac{\wp_Y}{\pro^{(H)}}\ci\frac{p_F}{G\ti H})$ is the quotient of
\begin{eqnarray*}
H\ti(F\un{X}{\ti}A)&=&\{(\eta,x,y,k,a)\in H\ti X\ti Y\ti K\ti A\mid \be(y)=k\al(x),\ x=\afr(a) \}\\
&\cong&\{(\eta,y,k,a)\in H\ti Y\ti K\ti A\mid \be(y)=k\al(\afr(x)) \}
\end{eqnarray*}
by the relation
\begin{eqnarray*}
&&(\eta,y,k,a)\sim (\eta\ppr,y\ppr,k\ppr,a\ppr)\\
&&\Longleftrightarrow\ \text{there exist}\ g\in G\ \text{and}\ h\in H\ \text{satisfying}\\
&&\qquad\qquad\eta=\eta\ppr h,\ y\ppr=hy,\ k\ppr=\tau_y(h)k\thh_{\afr(a)}(g)\iv,\ a\ppr=ga.
\end{eqnarray*}
$H$-action on $H\ti(F\un{X}{\ti}A)$ is given by
\[ h[\eta,y,k,a]=[h\eta,y,k,a]\quad(\fa h\in H, [\eta,y,k,a]\in\SIm(\frac{\wp_Y}{\pro^{(H)}}\ci\frac{p_F}{G\ti H})). \]

It can be easily confirmed that the map
\[ Y\un{Z}{\ti}\SIm(\althh\ci\frac{\afr}{G})\to\SIm(\frac{\wp_Y}{\pro^{(H)}}\ci\frac{p_F}{G\ti H})\ ;\ (y,[k,a])\mapsto [e,\afr(a),y,k,a] \]
is a well-defined isomorphism in $\Hs/Y$. This gives a natural isomorphism $(\bet)\uas\ci\atp\ov{\cong}{\Longrightarrow}(\frac{\wp_Y}{\pro^{(H)}})_+\ci(\frac{\wp_X}{\pro^{(G)}})\uas$.

\bigskip

{\rm (ii)} This follows from the isomorphism
\[ \ata\ci(\bet)_+\cong(\frac{\gamma}{\mu})_+\ci (\frac{\delta}{\nu})\uas \]
obtained in {\rm (i)} (applied to the transposed diagram), and the following adjointness.
\[ \ata\ci(\bet)_+\lt(\bet)\uas\ci\atb,\quad (\frac{\gamma}{\mu})_+\ci (\frac{\delta}{\nu})\uas\lt(\frac{\delta}{\nu})\bu\ci (\frac{\gamma}{\mu})\uas. \]
\end{proof}


Let us show that these properties induce (semi-)Mackey functors on $\Sbb$, in the sense of \cite[Definition 5.1.1]{N_BisetMackey}.

\begin{dfn}\label{DefOmega}
For any 0-cell $\xg$, define
\[ \Afr(\xg)=\Ob(\Gs/X)/\cong \]
to be the set of isomorphism classes of objects in $\Gs/X$. Coproducts and products (= fibered products over $X$ as $G$-sets) give a structure of a commutative semi-ring on $\Afr(\xg)$. This is called the {\it semi-Burnside ring}.

Define $\Om(\xg)$ to be the additive completion (= Grothendieck ring)
\[ \Om(\xg)=K_0(\Afr(\xg)) \]
of $\Afr(\xg)$. This is called the {\it Burnside ring}.

Remark that the functors $\ata,\atp,\atb$ associated to 1-cell $\oc$ induce maps on isomorphism classes
\begin{eqnarray*}
&\Afr\uas(\althh)\co\Afr(\yh)\to\Afr(\xg),&\\
&\Afr_+(\althh)\co\Afr(\xg)\to\Afr(\yh),&\\
&\Afr\bu(\althh)\co\Afr(\xg)\to\Afr(\yh)&
\end{eqnarray*}
respectively. By the adjoint property, $\Afr\uas(\althh)$ is a semi-ring homomorphism, $\Afr_+(\althh)$ is an additive monoid homomorphism, and $\Afr\bu(\althh)$ is a multiplicative monoid homomorphism.
\end{dfn}


The next corollary immediately follows from what we have shown.
\begin{cor}\label{CorOmega}
The triplet $(\Afr\uas,\Afr_+,\Afr\bu)$ satisfies the following.
\begin{enumerate}
\item $\Afr\uas\co\Sbb^{\op}\to\Sett$ and $\Afr_+,\Afr\bu\co\Sbb\to\Sett$ are strict 2-functors, where $\Sett$ is the category of sets, viewed as a 2-category equipped only with identity 2-cells. For any 0-cell $\xg\in\Sbb^0$, we have
\[ \Afr\uas(\xg)=\Afr_+(\xg)=\Afr\bu(\xg)\ (=\Afr(\xg)). \]
\item For any pair of 0-cells $\xg,\yh\in\Sbb^0$, if we take their bicoproduct
\[ \xg\ov{\frac{\ups_X}{\iota^{(G)}}}{\lra}\frac{\Ind_{\iota^{(G)}}X\am\Ind_{\iota^{(H)}}Y}{G\ti H}\ov{\frac{\ups_Y}{\iota^{(H)}}}{\lla}\yh, \]
then
\[ (\Afr\uas(\frac{\ups_X}{\iota^{(G)}}),\Afr\uas(\frac{\ups_Y}{\iota^{(H)}}))\co\Afr(\frac{\Ind^{(G)}X\am\Ind^{(H)}Y}{G\ti H})\to\Afr(\xg)\ti\Afr(\yh) \]
is a bijection.
\item For any bipullback
\begin{equation}\label{This2PB}
\xy
(-8,7)*+{\wl}="0";
(8,7)*+{\xg}="2";
(-8,-7)*+{\yh}="4";
(8,-7)*+{\zk}="6";
{\ar^{\frac{\gamma}{\mu}} "0";"2"};
{\ar_{\frac{\delta}{\nu}} "0";"4"};
{\ar^{\althh} "2";"6"};
{\ar_{\bet} "4";"6"};
{\ar@{<=}^{\kappa} (-2,-2);(2,2)};
\endxy
\end{equation}
in $\Sbb$, the following diagrams are commutative.
\[
\xy
(-11,7)*+{\Afr(\wl)}="0";
(11,7)*+{\Afr(\xg)}="2";
(-11,-7)*+{\Afr(\yh)}="4";
(11,-7)*+{\Afr(\zk)}="6";
{\ar^{\Afr_+(\frac{\gamma}{\mu})} "0";"2"};
{\ar^{\Afr\uas(\frac{\delta}{\nu})} "4";"0"};
{\ar_{\Afr\uas(\althh)} "6";"2"};
{\ar_{\Afr_+(\bet)} "4";"6"};
{\ar@{}|\circlearrowright "0";"6"};
\endxy
\ ,\quad
\xy
(-11,7)*+{\Afr(\wl)}="0";
(11,7)*+{\Afr(\xg)}="2";
(-11,-7)*+{\Afr(\yh)}="4";
(11,-7)*+{\Afr(\zk)}="6";
{\ar^{\Afr\bu(\frac{\gamma}{\mu})} "0";"2"};
{\ar^{\Afr\uas(\frac{\delta}{\nu})} "4";"0"};
{\ar_{\Afr\uas(\althh)} "6";"2"};
{\ar_{\Afr\bu(\bet)} "4";"6"};
{\ar@{}|\circlearrowright "0";"6"};
\endxy
\]
\end{enumerate}
\end{cor}

\begin{rem}\label{Rem22A}
As defined in \cite[Definition 5.1.1]{N_BisetMackey}, a pair of strict 2-functors $(S\uas,S\sas)$ satisfying the condition {\rm (1),(2),(3)} above, is called a {\it semi-Mackey functor} on $\Sbb$. Corollary \ref{CorOmega} is saying that $(\Afr\uas,\Afr_+)$ and $(\Afr\uas,\Afr\bu)$ are semi-Mackey functors.

If $(S\uas,S\sas)$ is a semi-Mackey functor on $\Sbb$, the value $S(\xg)=S\uas(\xg)=S\sas(\xg)$ has a naturally induced commutative monoid structure for each 0-cell $\xg$ (\cite[Proposition 5.1.7]{N_BisetMackey}). If this monoid is an abelian group for any $\xg$, then $(S\uas,S\sas)$ is called a {\it Mackey functor} on $\Sbb$.
\end{rem}


\section{Partial Tambara structure on $\Om$}

For a fixed finite group $G$, a Tambara functor is a machinery to deal with $\Ind,\Res$ and $\Jnd$ between subgroups of $G$. While the compositions $\Res\ci\Ind$ (and $\Res\ci\Jnd$) are calculated by using pullbacks, the compositions $\Jnd\ci\Ind$ are calculated by using {\it exponential diagrams}.
Let us briefly recall the definition of ordinary {\it Tambara functors} from \cite{Tambara}. (In \cite{Tambara}, Tambara functor is called a TNR-functor.) 

\begin{dfn}\label{DefExp}$($\cite[section 1]{Tambara}$)$
Let $G$ be a fixed finite group. For a morphism $f\in\Gs(X,Y)$ and an object $\Aa\in\Ob(\Gs/X)$, take
\begin{eqnarray*}
f\bu\Aa&=&(\Pi_f(A)\ov{\pi_Y}{\lra}Y),\\
f\uas f\bu\Aa&=&(\XY\Pi_f(A)\ov{p_X}{\lra}X).
\end{eqnarray*}
Then the unit of the adjoint $f\uas\lt f\bu$ gives a morphism $\rho\in\Gs/X(f\uas f\bu\Aa,\Aa)$. Thus we obtain a commutative diagram in $\Gs$
\begin{equation}\label{DiagExp}
\xy
(-18,8)*+{X}="0";
(-3,8)*+{A}="2";
(18,8)*+{\XY\Pi_f(A)}="4";
(-18,-8)*+{Y}="10";
(18,-8)*+{\Pi_f(A)}="14";
{\ar_{\afr} "2";"0"};
{\ar_(0.62){\rho} "4";"2"};
{\ar_{f} "0";"10"};
{\ar^{f\di} "4";"14"};
{\ar^{\pi_Y} "14";"10"};
{\ar@{}|\circlearrowright "0";"14"};
\endxy
\end{equation}
whose outer square is a fibered product. Explicitly, the map $\rho$ is given by
\[ \rho(x,(y,\sig))=\sig(x)\qquad(\fa (x,(y,\sig))\in\XY\Pi_f(A)). \]

A diagram in $\Gs$ isomorphic to $(\ref{DiagExp})$ arising from some $f$ and $\afr$, is called an {\it exponential diagram}.
\end{dfn}


\begin{dfn}\label{DefTam}$($\cite[section 2]{Tambara}$)$
Let $G$ be a fixed finite group. A {\it semi-Tambara functor} $T$ is a triplet $T=(T\uas,T_+,T\bu)$ of functors on $\Gs$
\begin{eqnarray*}
T\uas\co\Gs\to\Sett,&&\text{contravariant}\\
T_+\co\Gs\to\Sett,&&\text{covariant}\\
T\bu\co\Gs\to\Sett,&&\text{covariant}
\end{eqnarray*}
which satisfies the following conditions.
\begin{itemize}
\item[{\rm (i)}] $(T\uas,T_+)$ and $(T\uas,T\bu)$ are semi-Mackey functors on $G$. In particular we have
\[ T\uas(X)=T\bu(X)=T_+(X) \]
for any $X\in\Ob(\Gs)$. We write this simply as $T(X)$.
\item[{\rm (ii)}] For any exponential diagram
\[
\xy
(-18,8)*+{X}="0";
(-1,8)*+{A}="2";
(17,8)*+{Z}="4";
(-18,-8)*+{Y}="10";
(17,-8)*+{P}="14";
{\ar_{\afr} "2";"0"};
{\ar_{\rho} "4";"2"};
{\ar_{f} "0";"10"};
{\ar^{\lam} "4";"14"};
{\ar^{\pi} "14";"10"};
{\ar@{}|\circlearrowright "0";"14"};
\endxy
\]
in $\Gs$, the following diagram becomes commutative.
\[
\xy
(-18,8)*+{T(X)}="0";
(-1,8)*+{T(A)}="2";
(17,8)*+{T(Z)}="4";
(-18,-8)*+{T(Y)}="10";
(17,-8)*+{T(P)}="14";
{\ar_{T_+(\afr)} "2";"0"};
{\ar^{T\uas(\rho)} "2";"4"};
{\ar_{T\bu(f)} "0";"10"};
{\ar^{T\bu(\lam)} "4";"14"};
{\ar^{T_+(\pi)} "14";"10"};
{\ar@{}|\circlearrowright "0";"14"};
\endxy
\]
\end{itemize}

If $T$ is a semi-Tambara functor, then $T(X)$ has a canonically induced structure of a commutative semi-ring (\cite{Tambara}). For any $G$-map $f$, each $T\uas(f),T_+(f),T\bu(f)$ becomes semi-ring homomorphism, additive homomorphism, and multiplicative homomorphism, respectively.

A semi-Tambara functor $T$ is called a {\it Tambara functor} if $T(X)$ is a ring for any $X\in\Ob(\Gs)$. This is equivalent to require $(T\uas,T_+)$ to be Mackey functor on $G$.
\end{dfn}
\begin{ex}\label{SemiBurnTam}
For any fixed $G$, semi-Burnside rings form a semi-Tambara functor $A$ on $G$, which assigns
\begin{eqnarray*}
X\mapsto\ A(X)=(\Gs/X)/\cong\qquad\ &&(\fa X\in\Ob(\Gs)),\\
f\mapsto\begin{cases}A\uas(f)=f\uas\\ A_+(f)=f_+\\ A\bu(f)=f\bu\\ \end{cases}&&(\fa f\in\Gs(X,Y))
\end{eqnarray*}
as in \cite{Tambara}. Similarly, Burnside rings form a Tambara functor on $G$.
\end{ex}


\bigskip

We consider their analogs on $\Sbb$.
\begin{dfn}\label{DefPartialExp}
For a 1-cell $\oc$ and an object $\Aa\in\Ob(\Gs/X)$, take
\begin{eqnarray*}
\atb\Aa&=&(P\ov{\pi_Y}{\lra}Y),\\
\ata \atb\Aa&=&(\XY P\ov{p_X}{\lra}X).
\end{eqnarray*}
Then the unit of the adjoint $\ata\lt\atb$ gives a morphism
\[ \zeta\in\Gs/X(\ata\ \atb\Aa,\Aa). \]
Together with Proposition \ref{PropAst}, we obtain a commutative diagram in $\Sbb$
\begin{equation}\label{DiagPartialExp}
\xy
(-18,8)*+{\xg}="0";
(-3,8)*+{\frac{A}{G}}="2";
(18,8)*+{\frac{\XY P}{G}}="4";
(-18,-8)*+{\yh}="10";
(18,-8)*+{\frac{P}{H}}="14";
{\ar_{\frac{\afr}{G}} "2";"0"};
{\ar_(0.52){\frac{\zeta}{G}} "4";"2"};
{\ar_{\althh} "0";"10"};
{\ar^{\althhd} "4";"14"};
{\ar^{\frac{\pi_Y}{H}} "14";"10"};
{\ar@{}|\circlearrowright "0";"14"};
\endxy
\end{equation}
whose outer square is a bipullback. In this article, we call $(\ref{DiagPartialExp})$ the {\it partial exponential diagram} in $\Sbb$, associated to $\althh$ and $\afr$. The term {\it partial} means that, contrary to the case of ${}_G\mathit{set}$, 1-cells $\frac{\afr}{G}$ appearing in $(\ref{DiagPartialExp})$ is restricted to equivariant ones.
\end{dfn}


\begin{rem}\label{RemQuasiExp}
Let us look at the construction of $(\ref{DiagPartialExp})$ more in detail.
\begin{itemize}
\item[-] Put $\wt{X}=\SIm(\althh)$, and $\althh=\frac{\wt{\al}}{H}\ci\ulthh$ be the $\SIm$-factorization.
\item[-] Take $\Aa\uth=(A\uth,\afr\uth)$, and put $S_{\thh}(\Aa\uth)=(\wt{A}\ov{\wt{\afr}}{\lra}\wt{X})$.
\end{itemize}
Then by definition, we have
\begin{eqnarray*}
&\atb\Aa=\Pi_{\wt{\al}}(\wt{A},\wt{\afr})=(\Pi_{\wt{\al}}(\wt{A})\ov{\pi_Y}{\lra}Y),&\\
&\ata\atb\Aa=\ata\Pi_{\wt{\al}}(\wt{A},\wt{\afr})=(\XY\Pi_{\wt{\al}}(\wt{A})\ov{p_X}{\lra}X).&
\end{eqnarray*}
Here, $\Pi_{\wt{\al}}(\wt{A})$ is given by
\[ \Pi_{\wt{\al}}(\wt{A})=\Set{ (y,\sig)|\begin{array}{l}y\in Y,\\ \sig\in\sett(\wt{\al}\iv(y),\wt{A}),\end{array}%
\xy
(-10,6)*+{\wt{\al}\iv(y)}="0";
(10,6)*+{\wt{A}}="2";
(20.7,4.6)*+{=\HG A\uth}="3";
(0,-8)*+{\wt{X}}="4";
{\ar^(0.54){\sig} "0";"2"};
{\ar@{_(->} "0";"4"};
{\ar^(0.4){\wt{\afr}} "2";"4"};
{\ar@{}|\circlearrowright (0,-4);(0,6)};
\endxy} \]
Proposition \ref{PropAst} yields a bipullback in $\Sbb$ as follows.
\begin{equation}\label{PBUse1}
\xy
(-18,8)*+{\xg}="0";
(18,8)*+{\frac{\XY \Pi_{\wt{\al}}(\wt{A})}{G}}="4";
(-18,-8)*+{\yh}="10";
(18,-8)*+{\frac{\Pi_{\wt{\al}}(\wt{A})}{H}}="14";
{\ar_(0.56){\frac{p_X}{G}} "4";"0"};
{\ar_{\althh} "0";"10"};
{\ar^{\althhd} "4";"14"};
{\ar^(0.52){\frac{\pi_Y}{H}} "14";"10"};
{\ar@{}|\circlearrowright "0";"14"};
\endxy
\end{equation}


Explicitly, $\zeta\co \XY\Pi_{\wt{\al}}(\wt{A})\to A$ is given as follows. Let $(x,(y,\sig))\in \XY\Pi_{\wt{\al}}(\wt{A})$ be any element. Since it satisfies
\[ \wt{\al}([e,x])=\al(x)=\pi_Y(y,\sig)=y, \]
we can apply $\sig$ to $[e,x]\in\wt{\al}\iv(y)$, and express its value as
\begin{equation}\label{g1}
\sig([e,x])=[\eta,a]\quad\text{in}\ \HG A\uth
\end{equation}
with some $\eta\in H$ and $a\in A\uth$. Then the equality
\[ [e,x]=\wt{\afr}([\eta,a])=[\eta,\afr(a)]\quad\text{in}\ \wt{X} \]
implies the existence of $g\in G$ satisfying
\begin{equation}\label{g2}
\eta=\thh_{\afr(a)}(g)\quad\text{and}\quad x=g\afr(a).
\end{equation}
By using the definition of $A\uth\se A$, we can confirm that $a_0=ga\in A\uth$ does not depend on choices of $\eta,a,g$ satisfying $(\ref{g1}),(\ref{g2})$. In fact, $a_0$ is uniquely determined by the equalities
\begin{equation}\label{g3}
[e,a_0]=\sig([e,x])\quad\text{and}\quad \afr(a_0)=x.
\end{equation}
Thus the assignment $\zeta\ppr(x,(y,\sig))=a_0$ gives a well-defined map
\[ \zeta\ppr\co\XY\Pi_{\wt{\al}}(\wt{A})\to A\uth, \]
which is shown to be $G$-equivariant, by using the characterizing equality $(\ref{g3})$. We obtain $\zeta$ as the composition of $\XY\Pi_{\wt{\al}}(\wt{A})\ov{\zeta\ppr}{\lra}A\uth\hookrightarrow A$.
Thus there is a commutative diagram
\[
\xy
(-12,8)*+{A\uth}="0";
(12,8)*+{\XY\Pi_{\wt{\al}}(\wt{A})}="2";
(-4,-4)*+{}="3";
(-12,-8)*+{X}="4";
(3,4)*+{}="5";
(12,-8)*+{A}="6";
{\ar_(0.6){\zeta\ppr} "2";"0"};
{\ar_{\afr\uth} "0";"4"};
{\ar@{_(->} "0";"6"};
{\ar^{\zeta} "2";"6"};
{\ar^{\afr} "6";"4"};
{\ar@{}|\circlearrowright "2";"3"};
{\ar@{}|\circlearrowright "4";"5"};
\endxy
\]
in $\Gs$.


From $\wt{\al}\in\Hs(\wt{X},Y)$ and $(\wt{A},\wt{\afr})\in\Ob(\Hs/Y)$, we can draw the associated exponential diagram
\begin{equation}\label{PBUse2}
\xy
(-20,8)*+{\wt{X}}="0";
(-3,8)*+{\wt{A}}="2";
(1,17)*+{}="3";
(20,8)*+{\wt{X}\un{Y}{\ti}\Pi_{\wt{\al}}(\wt{A})}="4";
(-20,-8)*+{Y}="10";
(20,-8)*+{\Pi_{\wt{\al}}(\wt{A})}="14";
{\ar_{\wt{\afr}} "2";"0"};
{\ar_(0.62){\rho} "4";"2"};
{\ar_{\wt{\al}} "0";"10"};
{\ar^{\wt{\al}\di} "4";"14"};
{\ar^{\pi_Y} "14";"10"};
{\ar@/_1.80pc/_{p_{\wt{X}}} "4";"0"};
{\ar@{}|\circlearrowright "0";"14"};
{\ar@{}|\circlearrowright "2";"3"};
\endxy
\end{equation}
in $\Hs$. Then we have the following commutative diagram in $\Sbb$,
\begin{equation}\label{PBUse3}
\xy
(-14,18)*+{\xg}="0";
(14,18)*+{\frac{\XY\Pi_{\wt{\al}}(\wt{A})}{G}}="2";
(-25,0)*+{}="3";
(-14,0)*+{\frac{\wt{X}}{H}}="4";
(14,0)*+{\frac{\wt{X}\un{Y}{\ti}\Pi_{\wt{\al}}(\wt{A})}{H}}="6";
(31,0)*+{}="7";
(-14,-18)*+{\yh}="8";
(14,-18)*+{\frac{\Pi_{\wt{\al}}(\wt{A})}{H}}="10";
{\ar_(0.56){\frac{p_X}{G}} "2";"0"};
{\ar^{\ulthh} "0";"4"};
{\ar_{\frac{\ups_{\al\di}}{\thh\di}} "2";"6"};
{\ar_(0.56){p_{\wt{X}}} "6";"4"};
{\ar^{\frac{\wt{\al}}{H}} "4";"8"};
{\ar_{\frac{\wt{\al}\di}{H}} "6";"10"};
{\ar^{\pi_Y} "10";"8"};
{\ar@/_2.40pc/_{\althh} "0";"8"};
{\ar@/^2.80pc/^{\althhd} "2";"10"};
{\ar@{}|\circlearrowright "0";"6"};
{\ar@{}|\circlearrowright "4";"10"};
{\ar@{}|\circlearrowright "3";"4"};
{\ar@{}|\circlearrowright "6";"7"};
\endxy
\end{equation}
where $\ups_{\al\di}$ is defined by $\ups_{\al\di}(x,(y,\sig))=([e,x],(y,\sig))$.

In $(\ref{PBUse3})$, the outer square is equal to $(\ref{PBUse1})$, which is a bipullback. The lower square comes from the outer square of $(\ref{PBUse2})$, which is also a bipullback. Thus the upper square of $(\ref{PBUse3})$ becomes a bipullback. In particular the stab-surjectivity of $\ulthh$ implies that $\frac{\ups_{\al\di}}{\thh\di}$ is stab-surjective. Thus $\althhd=\frac{\wt{\al}\di}{H}\ci\frac{\ups_{\al\di}}{\thh\di}$ gives a factorization of
$\althhd$ into equivariant $\frac{\wt{\al}\di}{H}$ and stab-surjective $\frac{\ups_{\al\di}}{\thh\di}$, which implies the following $H$-isomorphism.
\begin{eqnarray*}
\wt{X}\un{Y}{\ti}\Pi_{\wt{\al}}(\wt{A})&\cong&\SIm(\althhd)\\
&=&\SIm(\pi_Y\ci\althhd)\ =\ \SIm(\althh\ci\frac{p_X}{G}).
\end{eqnarray*}

If we apply Lemma \ref{LemRightAdj2} {\rm (2)} to the stab-surjective 1-cell $\ulthh$, then the unit of $(\ulthh)_+\lt(\ulthh)\uas$ gives an isomorphism
\[ \om\co(A\uth,\afr\uth)\to (\XY\wt{A},p_X)\ ;\ a\mapsto (\afr(a),[e,a]). \]
Composing this with the diagram obtained in Proposition \ref{PropAst},
we have a bipullback in $\Sbb$
\[
\xy
(-7,6)*+{\xg}="0";
(7,6)*+{\frac{A\uth}{G}}="2";
(-7,-6)*+{\frac{\wt{X}}{H}}="4";
(7,-6)*+{\frac{\wt{A}}{H}}="6";
{\ar_{\frac{\afr\uth}{G}} "2";"0"};
{\ar_{\ulthh} "0";"4"};
{\ar^{\frac{\ups}{\thh\ppr}} "2";"6"};
{\ar^{\frac{\wt{\afr}}{H}} "6";"4"};
{\ar@{}|\circlearrowright "0";"6"};
\endxy
\]
where $\frac{\ups}{\thh\ppr}$ is defined by
\[ \ups(a)=[e,a],\quad\thh\ppr_a=\thh_{\afr(a)} \]
for any $a\in A\uth$. This $\frac{\ups}{\thh\ppr}$ makes the following diagram commutative.
\begin{equation}\label{Concat}
\xy
(-18,8)*+{\xg}="0";
(0,8)*+{\frac{A\uth}{G}}="2";
(22,8)*+{\frac{\XY\Pi_{\wt{\al}}(\wt{A})}{G}}="4";
(-18,-8)*+{\frac{\wt{X}}{H}}="10";
(0,-8)*+{\frac{\wt{A}}{H}}="12";
(22,-8)*+{\frac{\wt{X}\un{Y}{\ti}\Pi_{\wt{\al}}(\wt{A})}{H}}="14";
{\ar_{\frac{\afr\uth}{G}} "2";"0"};
{\ar_(0.56){\frac{\zeta\ppr}{G}} "4";"2"};
{\ar_{\ulthh} "0";"10"};
{\ar^{\frac{\ups}{\thh\ppr}} "2";"12"};
{\ar^{\frac{\ups_{\al\di}}{\thh\di}} "4";"14"};
{\ar^{\frac{\wt{\afr}}{H}} "12";"10"};
{\ar^(0.56){\frac{\rho}{H}} "14";"12"};
{\ar@{}|\circlearrowright "0";"12"};
{\ar@{}|\circlearrowright "2";"14"};
\endxy
\end{equation}
Since the left and outer squares are bipullbacks, the right square also becomes a bipullback.
\end{rem}


\begin{prop}\label{PropSemiTam}
Let $\Afr=(\Afr\uas,\Afr_+,\Afr\bu)$ be the triplet as in Definition \ref{DefOmega}. For any partial exponential diagram $(\ref{DiagPartialExp})$, the following diagram becomes commutative.
\[
\xy
(-20,8)*+{\Afr(\xg)}="0";
(-1,8)*+{\Afr(\frac{A}{G})}="2";
(22,8)*+{\Afr(\frac{\XY P}{G})}="4";
(-20,-8)*+{\Afr(\yh)}="10";
(22,-8)*+{\Afr(\frac{P}{H})}="14";
{\ar_{\Afr_+(\frac{\afr}{G})} "2";"0"};
{\ar^(0.46){\Afr\uas(\frac{\zeta}{G})} "2";"4"};
{\ar_{\Afr\bu(\althh)} "0";"10"};
{\ar^{\Afr\bu(\althhd)} "4";"14"};
{\ar^{\Afr_+(\frac{\pi_Y}{H})} "14";"10"};
{\ar@{}|\circlearrowright "0";"14"};
\endxy
\]
\end{prop}
\begin{proof}
Let $\Bb\in\Ob(\Gs/A)$ be any object. It suffices to give an isomorphism
\[ \atb(\frac{\afr}{G})_+\Bb\cong(\frac{\pi_Y}{H})_+(\althhd)\bu(\frac{\zeta}{G})\uas\Bb \]
in $\Hs/Y$. 

We use the detailed description given in Remark \ref{RemQuasiExp}.
Applying $(-)\uth$ to the morphism
\[ \bfr\co\afr_+\Bb=(B,\afr\ci\bfr)\to\Aa \]
in $\Gs/X$, we obtain
\[ \bfr\uth\co(B\uth,(\afr\ci\bfr)\uth)\to(A\uth,\afr\uth). \]
Furthermore, applying $S_{\thh}\co\Gs/X\to\Hs/\wt{X}$, we obtain a morphism
\[ \wt{\bfr}\co(\wt{B},\wt{(\afr\ci\bfr)})\to(\wt{A},\wt{\afr}), \]
where we put $S_{\thh}(B\uth,(\afr\ci\bfr)\uth)=(\wt{B},\wt{\bfr})$ and $S_{\thh}(\bfr\uth)=\wt{\bfr}$.

If we regard $(B\uth,\bfr\uth)$ as an object in $\Gs/A\uth$, then we have
\begin{eqnarray*}
&(-)\uth\ci(\frac{\afr}{G})_+\Bb=(B\uth,(\afr\ci\bfr)\uth)=(\frac{\afr\uth}{G})_+(B\uth,\bfr\uth),&\\
&(\frac{\ups}{\thh\ppr})_+(B\uth,\bfr\uth)\cong(\wt{B},\wt{\bfr}).&
\end{eqnarray*}
Moreover, the bipullback in $(\ref{Concat})$ gives isomorphisms
\begin{eqnarray*}
S_{\thh}\ci(\frac{\afr\uth}{G})_+(B\uth,\bfr\uth)&\cong&(\ulthh)_+(\frac{\afr\uth}{G})_+(B\uth,\bfr\uth)\\
&\cong&(\frac{\wt{\afr}}{H})_+(\frac{\ups}{\thh\ppr})_+(B\uth,\bfr\uth)\ \cong\ \wt{\afr}_+(\wt{B},\wt{\bfr}),
\end{eqnarray*}
\begin{eqnarray*}
S_{\althhd}\ci(\frac{\zeta\ppr}{G})\uas(B\uth,\bfr\uth)&\cong&(\frac{\ups_{\al\di}}{\thh\di})_+(\frac{\zeta\ppr}{G})\uas(B\uth,\bfr\uth)\\
&\cong&(\frac{\rho}{H})\uas(\frac{\ups}{\thh\ppr})_+(B\uth,\bfr\uth)\ \cong\ \rho\uas(\wt{B},\wt{\bfr})
\end{eqnarray*}
by Corollary \ref{CorOmega}.

Since $(\ref{PBUse2})$ is an exponential diagram in $\Hs$, by Example \ref{SemiBurnTam} we have
\[ \Pi_{\wt{\al}}\wt{\afr}_+(\wt{B},\wt{\bfr})\cong\pi_{Y_+}\Pi_{\wt{\al}\di}\rho\uas(\wt{B},\wt{\bfr})\ \ \text{in}\ \Hs, \]
for $(\wt{B},\wt{\bfr})\in\Ob(\Hs/\wt{A})$.

Composing the isomorphisms so far obtained, we have
\begin{eqnarray*}
\atb(\frac{\afr}{G})_+\Bb&\cong&\Pi_{\wt{\al}}S_{\thh}\ci(-)\uth\ci(\frac{\afr}{G})_+\Bb\\
&\cong&\Pi_{\wt{\al}}S_{\thh}(\frac{\afr\uth}{G})_+(B\uth,\bfr\uth)\\
&\cong&\Pi_{\wt{\al}}\wt{\afr}_+(\wt{B},\wt{\bfr})\\
&\cong&\pi_{Y+}\Pi_{\wt{\al}\di}\rho\uas(\wt{B},\wt{\bfr})\\
&\cong&(\frac{\pi_Y}{H})_+\Pi_{\wt{\al}\di} S_{\althhd}\ci\zeta^{\prime\ast}(B\uth,\bfr\uth).
\end{eqnarray*}
Thus it remains to show
\begin{equation}\label{ToShow}
\zeta^{\prime\ast}(B\uth,\bfr\uth)\cong(-)^{\thh\di}\ci\zeta\uas\Bb.
\end{equation}
Take a fibered product
\[
\xy
(-12,8)*+{B}="0";
(12,8)*+{B\un{A}{\ti}(X\un{Y}{\ti}\Pi_{\wt{\al}}(\wt{A}))}="2";
(29.8,8.9)*+{=C}="3";
(-12,-7)*+{A}="4";
(12,-7)*+{X\un{Y}{\ti}\Pi_{\wt{\al}}(\wt{A})}="6";
{\ar_(0.73){p} "2";"0"};
{\ar_{\bfr} "0";"4"};
{\ar^{q} "2";"6"};
{\ar^(0.6){\zeta} "6";"4"};
{\ar@{}|\circlearrowright "0";"6"};
\endxy
\]
in $\Gs$. Then we have
\begin{eqnarray*}
\zeta\uas\Bb&=&(C,q),\\
\zeta^{\prime\ast}(B\uth,\bfr\uth)&=&\zeta\uas(B\uth\ov{\bfr|_{B\uth}}{\lra}A)\ \cong\ (p\iv(B\uth),q\ppr)
\end{eqnarray*}
where $q\ppr\co p\iv(B\uth)\to\XY\Pi_{\wt{\al}}(\wt{A})$ is the restriction of $q$ onto $p\iv(B\uth)\se C$. This is the left hand side of $(\ref{ToShow})$. On the other hand, the right hand side of $(\ref{ToShow})$ is, by definition
\[ (-)^{\thh\di}\ci\zeta\uas\Bb=(C^{\thh\di},q\pprr), \]
where $q\pprr$ is similarly the restriction of $q$ onto $C^{\thh\di}$.
Thus it suffices to show $p\iv(B\uth)=C^{\thh\di}$ as subsets of $C$.


Explicitly, by the characterization $(\ref{g3})$, $C$ is
\begin{eqnarray*}
C&=&\{(b,x,(y,\sig))\in B\ti \XY\Pi_{\wt{\al}}(\wt{A}) \mid \bfr(b)=\zeta(x,(y,\sig)) \}\\
&=&\{ (b,x,(y,\sig))\in B\ti \XY\Pi_{\wt{\al}}(\wt{A}) \mid [e,\bfr(b)]=\sig([e,x]), \afr(\bfr(b))=x \},
\end{eqnarray*}
on which $G$ acts by
\begin{eqnarray*}
&g(b,x,(y,\sig))=(gb,gx,\thh_x(g)(y,\sig))&\\
&(\fa g\in G, (b,x,(y,\sig))\in C).&
\end{eqnarray*}
Maps $p$ and $q$ are defined by
\[ p(b,x,(y,\sig))=b   ,\quad q(b,x,(y,\sig))=(x,(y,\sig)) \]
for any $(b,x,(y,\sig))\in C$.

Let $c=(b,x,(y,\sig))\in C$ be any element.
Remark that we have
\[ \thh_{\afr(\bfr(b))}=\thh_x=\thh\di_{(x,(y,\sig))}=\thh\di_{q(c)}. \]
Thus for any $g,g\ppr\in G$, we have
\begin{eqnarray*}
g\un{(\thh,\afr\ci\bfr)}{\eq}g\ppr\ \text{at}\ b&\LR& \thh_x(g)=\thh_x(g\ppr)\quad\text{and}\quad gx=g\ppr x\\
&\LR& \thh_x(g)=\thh_x(g\ppr)\quad\text{and}\quad g(x,(y,\sig))=g\ppr (x,(y,\sig))\\
&\LR& g\un{(\thh\di,q)}{\eq}g\ppr\ \text{at}\ c.
\end{eqnarray*}
This implies
\begin{eqnarray*}
c\in p\iv(B\uth)&\LR&b\in B\uth\\
&\LR&gb=g\ppr b\ \ \text{whenever}\ g\un{(\thh,\afr\ci\bfr)}{\eq}g\ppr\ \text{at}\ b\\
&\LR&g(b,x,(y,\sig))=g\ppr (b,x,(y,\sig))\ \ \text{whenever}\ g\un{(\thh\di,q)}{\eq}g\ppr\ \text{at}\ c\\
&\LR&c\in C^{\thh\di},
\end{eqnarray*}
which means $p\iv(B\uth)=C^{\thh\di}$.
\end{proof}

From those obtained so far, we can show the triplet $\Om=(\Om\uas,\Om_+,\Om\bu)$ satisfies analogous properties to Tambara functors.
We use the definition and properties of {\it polynomial maps} (which are called {\it algebraic maps} in \cite{D-S}).

\begin{dfn}\label{DefPoly}
(\cite[Section 5.6]{D-S})
Let $A$ be an additive monoid, $M$ be an abelian group, and $\vp\co A \to M$ be a map. For any $n$ elements $a_1,a_2,\ldots ,a_n\in A$, a map $D_{(a_1,a_2,\ldots , a_n)}\vp\co A\to M$ is defined by
\[
D_{(a_1,\ldots , a_n)}\vp(x)=\sum _{k=0}^n\big(\sum _{1\le i_1<\cdots <i_k\le n}(-1)^{n-k}\vp (x+a_{i_1}+\ldots +a_{i_k})\big).
\]
The map $\vp$ is said to be {\it polynomial} if either $\vp=0$ or there exists a positive integer $n$ such that
\[
D_{(a_1,\ldots , a_n)}\vp=0\quad (\forall a_1,\ldots ,a_n\in A).
\] 

For any polynomial map $\vp\co A\to M$, its degree is defined by
\[
\deg \vp=\max \{n\in \mathbb{N}_{\ge 0}\mid \exists a_1,\ldots ,a_n \in A\ \text{such that}\ D_{(a_1,\ldots , a_n)}\vp\ne0 \}
\]
if $\vp\ne0$, and $\deg \vp=-1$ if $\vp=0$.
\end{dfn}

Also, for a map between additive monoids, we define as follows.
\begin{dfn}
A map $\vp\co A\to B$ between additive monoids $A$ and $B$ is said to be polynomial of degree $n$ if $\kappa_B\circ\vp\co A\to K_0B$ is polynomial of degree $n$ in the sense of Definition \ref{DefPoly}, where $\kappa_B\co B\to K_0B$ is the group completion map.
\end{dfn}

\begin{rem}\label{p2-B'}
Let $A,B,C$ be additive monoids. If $\vp\co A\to B$ and $\psi\co B\to C$ are polynomial maps of degree $m$ and $n$ respectively, then $\psi\circ\vp\co A\to C$ becomes polynomial of degree $\le mn$.
\end{rem}

\begin{prop}\label{p2-C'}
Let $\vp\co A\to B$ be a polynomial map between additive monoids $A$ and $B$. Then there exists a unique extension $\wt{\vp}\co K_0A\to K_0B$ of $\vp$ as a polynomial map, i.e. unique polynomial map $\wt{\vp}$ satisfying $\wt{\vp}\circ\kappa_A=\kappa_B\circ\vp$.
Moreover, the following holds.
\begin{enumerate}
\item If $\vp$ is additive, then $\wt{\vp}$ is also additive.
\item If $A,B$ are commutative semi-rings and $\vp$ is multiplicative, then $\wt{\vp}$ is also multiplicative.
\end{enumerate}
\end{prop}
\begin{proof}
The construction can be found in \cite{D-S}.
\end{proof}

\bigskip

\begin{prop}\label{OmegaAlg}
For any 1-cell $\oc$, those maps
\[ \Afr\uas(\althh),\Afr_+(\althh),\Afr\bu(\althh) \]
are polynomial. Thus they extend to yield the maps between $\Om(\xg)$ and $\Om(\yh)$
\[ \Om\uas(\althh),\Om_+(\althh),\Om\bu(\althh), \]
uniquely as polynomial maps.
\end{prop}
\begin{proof}
$\Afr\uas(\althh)$ and $\Afr_+(\althh)$ are additive, and hence polynomial. As for $\Afr\bu(\althh)$, by definition, it is a composition of maps induced from the functors
\begin{equation}\label{Last1}
S_{\thh}\ci(-)\uth\co\Gs/X\to\Gs/\SIm(\althh)
\end{equation}
and 
\begin{equation}\label{Last2}
\Pi_{\wt{\al}}\co\Gs/\SIm(\althh)\to\Hs/Y.
\end{equation}
As in \cite{Tambara}, the functor $(\ref{Last2})$ induces a polynomial map $\Afr\bu(\frac{\wt{\al}}{H})=A(\wt{\al})\co A(\SIm(\althh))\to A(Y)$. Since $S_{\thh}$ and $(-)\uth$ preserve coproducts, $(\ref{Last1})$ induces an additive map $\Afr(\xg)\to\Afr(\frac{\SIm(\althh)}{H})$.
\end{proof}

\begin{cor}
The triplet $\Om=(\Om\uas,\Om_+,\Om\bu)$ in Definition \ref{DefOmega} satisfies the following.
\begin{itemize}
\item[{\rm (i)}] $(\Om\uas,\Om_+)$ is a Mackey functor, and $(\Om\uas,\Om\bu)$ is a semi-Mackey functor on $\Sbb$.
\item[{\rm (ii)}] For any partial exponential diagram $(\ref{DiagPartialExp})$, the following diagram becomes commutative.
\[
\xy
(-20,8)*+{\Om(\xg)}="0";
(-1,8)*+{\Om(\frac{A}{G})}="2";
(22,8)*+{\Om(\frac{\XY P}{G})}="4";
(-20,-8)*+{\Om(\yh)}="10";
(22,-8)*+{\Om(\frac{P}{H})}="14";
{\ar_{\Om_+(\frac{\afr}{G})} "2";"0"};
{\ar^(0.46){\Om\uas(\frac{\zeta}{G})} "2";"4"};
{\ar_{\Om\bu(\althh)} "0";"10"};
{\ar^{\Om\bu(\althhd)} "4";"14"};
{\ar^{\Om_+(\frac{\pi_Y}{H})} "14";"10"};
{\ar@{}|\circlearrowright "0";"14"};
\endxy
\]
\end{itemize}
\end{cor}
\begin{proof}
This follows immediately from Remark \ref{Rem22A}, Propositions \ref{PropSemiTam}, \ref{p2-C'} and \ref{OmegaAlg}.
\end{proof}

\section*{Acknowledgement}
This article has been written when the author was staying at LAMFA, l'Universit\'{e} de Picardie-Jules Verne, by the support of JSPS Postdoctoral Fellowships for Research Abroad. He wishes to thank the hospitality of Professor Serge Bouc, Professor Radu Stancu and the members of LAMFA.

\end{document}